\documentclass[11pt]{amsart}
\usepackage{enumitem,amssymb,cmap,stmaryrd}

\newtheorem{theorem}{Theorem}[section]

\newtheorem{prop}[theorem]{Proposition}
\newtheorem{claim}{Claim}[theorem]
\newtheorem{subclaim}{Subclaim}[claim]
\newtheorem{lemma}[theorem]{Lemma}
\newtheorem{cor}[theorem]{Corollary}

\theoremstyle{definition} 
\newtheorem{definition}[theorem]{Definition}

\newtheorem{conv}[theorem]{Convention}

\theoremstyle{remark}
\newtheorem{remark}[theorem]{Remark}
\newtheorem{example}[theorem]{Example}

\providecommand{\myceil}[2]{\left\lceil #1 \right\rceil^{#2} }
\newcommand*\axiomfont[1]{\textsf{\textup{#1}}}

\newcounter{condition}

\newcommand{\one}{\mathop{1\hskip-3pt {\rm l}}}
\newcommand{\cone}[1]{{\mathbb P\mathrel{\downarrow}#1}}
\newcommand{\conea}[1]{{\mathbb A\mathrel{\downarrow}#1}}
\newcommand{\fork}[2]{{\pitchfork_{#1}}(#2)}
\newcommand{\diagonal}{\bigtriangleup}
\newcommand{\zfc}{\axiomfont{ZFC}}
\newcommand{\ch}{\axiomfont{CH}}
\renewcommand{\restriction}{\mathbin\upharpoonright}
\renewcommand{\mid}{\mathrel{|}\allowbreak}
\def\br{\blacktriangleright}
\def\sq{\sqsubseteq}
\def\s{\subseteq}
\DeclareMathOperator{\cl}{cl}
\DeclareMathOperator{\otp}{otp}
\DeclareMathOperator{\dom}{dom}
\DeclareMathOperator{\rng}{Im}
\DeclareMathOperator{\acc}{acc}
\DeclareMathOperator{\tr}{Tr}
\DeclareMathOperator{\cf}{cf}
\DeclareMathOperator{\cov}{cov}
\DeclareMathOperator{\refl}{Refl}
\DeclareMathOperator{\ord}{Ord}
\DeclareMathOperator{\add}{Add}

\DeclareMathOperator{\stem}{stem}
\newcommand{\lh}{\ell}

\title[Sigma-Prikry forcing I]{Sigma-Prikry forcing I:\\The Axioms}

\author{Alejandro Poveda}
\thanks{Poveda was partially supported by the Spanish Government under grant MTM2017-86777-P, by Generalitat de Catalunya (Catalan Government) under grant SGR 270-2017 and by MECD Grant FPU15/00026.}
\address{Departament de Matem\`atiques i Inform\`atica, Universitat de Barcelona. Gran Via de les Corts Catalanes, 585, 08007 Barcelona, Catalonia.}

\author{Assaf Rinot}
\thanks{Rinot was partially supported by the European Research Council (grant agreement ERC-2018-StG 802756) and by the Israel Science Foundation (grant agreement 2066/18).}
\address{Department of Mathematics, Bar-Ilan University, Ramat-Gan 5290002, Israel.} \urladdr{http://www.assafrinot.com}

\author{Dima Sinapova}
\thanks{Sinapova was partially supported by the National Science Foundation, Career-1454945.}
\address{Department of Mathematics, Statistics, and Computer Science\\ University of Illinois at Chicago\\ Chicago, IL 60607-7045\\ USA} \urladdr{https://homepages.math.uic.edu/~sinapova/}

\begin{document}
\begin{abstract}  We introduce a class of notions of forcing which we call $\Sigma$-Prikry,
and show that many of the known Prikry-type notions of forcing that centers around singular cardinals of countable cofinality
are $\Sigma$-Prikry. We show that given a $\Sigma$-Prikry poset $\mathbb P$ and a name for a non-reflecting stationary set $T$,
there exists a corresponding $\Sigma$-Prikry poset that projects to $\mathbb P$ and kills the stationarity of $T$.
Then, in a sequel to this paper, we develop an iteration scheme for $\Sigma$-Prikry posets.
Putting the two works together, we obtain a proof of the following.

{\bf Theorem.} If $\kappa$ is the limit of a countable increasing sequence of supercompact cardinals,
then there exists a cofinality-preserving forcing extension
in which $\kappa$ remains a strong limit,
every finite collection of stationary subsets of $\kappa^+$ reflects simultaneously, and $2^\kappa=\kappa^{++}$.
\end{abstract}
\date{\today}
\maketitle
\tableofcontents
\section{Introduction}

In \cite{cohen1,cohen2}, Cohen invented the method of forcing as a mean to prove the independence of mathematical propositions from $\zfc$ (the Zermelo-Fraenkel axioms for set theory).
With this method, one starts with an arbitrary (transitive) model $\mathbb M$ of $\zfc$, define there a partial order $\mathbb P$,
and then pass to a \emph{forcing extension} $\mathbb M[G]$ in which a new \emph{$\mathbb P$-generic} set $G$ is adjoined.
The outcome $\mathbb M[G]$ is the smallest model of $\zfc$ to contain all the elements of $\mathbb M$, as well as the object $G$.
For instance, in Cohen's celebrated work on the Continuum Hypothesis ($\ch$, asserting that $2^{\aleph_0}=\aleph_1$), he takes $\mathbb M$ to be G\"odel's  model \cite{godel} of $\zfc+\ch$,
and defines $\mathbb P$ in a way that ensures that any $\mathbb P$-generic set $G$ will consist of $\aleph_2$ many distinct real numbers.
Finally, to verify that ``$2^{\aleph_0}\geq\aleph_2$" indeed holds in $\mathbb M[G]$, Cohen proves that $\aleph_2$,
the second uncountable cardinal of $\mathbb M$, remains the second uncountable cardinal of $\mathbb M[G]$.
In fact, Cohen proves that $\mathbb P$ satisfies the \emph{countable chain condition} ($ccc$) and shows that this condition ensures that the cardinals structure of $\mathbb M[G]$ is identical to that of $\mathbb M$.

Now, let us consider a proposition $\varphi$ slightly more involved than $\ch$, say, $\varphi$ is of the form ``every uncountable group having property $p$, has property $q$, as well''.
Suppose that $\mathbb M$ is a model in which there is an uncountable group $A$ that forms a counterexample to $\varphi$.
Then we could try to cook up a poset $\mathbb P_A$ such that for any $\mathbb P_A$-generic set $G$, either $G$ witness in $\mathbb M[G]$ that $A$ has property $q$,
or $G$ witnesses in $\mathbb M[G]$ that $A$ ceased to have property $p$. This will solve our problem $\varphi$ for $A$,
but it is very likely that in our new model $\mathbb M[G]$ there are other (possibly new) counterexamples to $\varphi$, meaning that we need to fix yet another counterexample $A'$
and pass to a forcing extension $\mathbb M[G][H]$ solving the problem for $A'$, and basically ``keep going''. But will we ever catch our tail?

It is clear that to have a chance to catch our tail,
there is a need for a transfinite forcing iteration. However, unless various conditions are met, such a forcing iteration will ruin the cardinals structure,
leading to a meaningless solution of the problem $\varphi$, in the sense that all uncountable groups from the intermediate models will become countable at the final model.

The first successful transfinite iteration scheme was devised by Solovay and Tennenbaum in \cite{MR0294139},
who solved a problem concerning a particular type of linear orders of size $\aleph_1$ known as \emph{Souslin lines}.
They found a natural $ccc$ poset $\mathbb P_L$ to ``kill'' a given Souslin line $L$, proved that a (finite-support) iteration of $ccc$ posets is again $ccc$,
and proved that in an iteration of length $\aleph_2$, any Souslin line in the final model must show up in one of the intermediate models,
meaning that they can ensure that, in their final model, there are no Souslin lines.

The Solovay-Tennenbaum technique is very useful (see \cite{MR780933}), but it admits no  generalizations that allow to tackle problems concerning objects of size $>\aleph_1$.
One crucial reason for the lack of generalizations has to do with the poor behavior of the higher analogues of $ccc$ at the level of cardinals $>\aleph_1$ (see \cite{paper18,paper34,roslanowski} for a discussion and counterexamples).

Still, various iteration schemes for posets having strong forms of the $\kappa^{+}$-chain-condition for $\kappa$ regular were devised in
\cite{sh:80,sh:587,RoSh:655,MR2029324,RoSh:888,RoSh:942,RoSh:1001}.
In contrast, there is a dearth of works involving iterations at the level of the successor of singular cardinals.

A few ad-hoc treatments of iterations that are centered around a singular cardinal may be found in \cite[\S2]{sh:186}, \cite[\S10]{cfm} and \cite[\S1]{paper08},
and a  more general framework is offered by \cite[\S3]{sh:667}.
In \cite{DjSh:659}, the authors took another approach in which they first pursue a forcing iteration along a successor of a regular cardinal $\kappa$,
and at the very end they singularize $\kappa$ by appealing to Prikry forcing. This was then generalized to Radin forcing in \cite{CDMMSh:963}.

In this project, we propose yet another approach, allowing to put the Prikry-type forcing at $\kappa$ as our very first step of the iteration,
and then continue up to length $\kappa^{++}$ without collapsing cardinals.
We do so by identifying a class of Prikry-type posets that are iterable in a sense to be made precise.
The class is called \emph{$\Sigma$-Prikry},
where $\Sigma=\langle\kappa_n\mid n<\omega\rangle$ is a non-decreasing sequence of regular uncountable cardinals,
converging to our cardinal $\kappa$.
A member of the $\Sigma$-Prikry class is a triple $(\mathbb P,\lh,c)$ satisfying, among other things, the following:
\begin{itemize}
\item $\mathbb P=(P,\le)$ is a notion of forcing;
\item $\one_{\mathbb P}$ decides the value of $\kappa^+$ to be some cardinal $\mu$;
\item $\lh:P\rightarrow\omega$ is a monotone grading function;
\item $c:P\rightarrow \mu$ is a function witnessing that $\mathbb P$ is $\mu^+$-2-linked;
\item $(\mathbb P,\lh)$ has the Complete Prikry Property.
\end{itemize}

Here, \emph{$\mu^+$-2-linked} is a well-known strong form of the $\mu^+$-chain-condition;
as explained earlier, the latter would be too weak for any viable iteration scheme.
In contrast, the \emph{Complete Prikry Property}  is a new concept that we introduce here in order to simultaneously
capture two characteristic features of Prikry-type forcing: the \emph{decision by pure extension property} and the \emph{strong Prikry property}.
The exact definition of $\Sigma$-Prikry may be found in Section~\ref{SPS} and a list of examples is given in Section~\ref{examples}.

Now, let us describe the first application of our framework.
In his dissertation \cite{AS}, Sharon claimed that if $\kappa$ is the limit of a strictly increasing sequence $\langle \kappa_n\mid n<\omega\rangle$ of supercompact cardinals,
then, in some cardinals-preserving forcing extension, $\kappa$ remains a strong limit, $2^\kappa=\kappa^{++}$, and every stationary subset of $\kappa^+$ reflects.
Sharon's model is obtained by first blowing up the power of $\kappa$ using the forcing of \cite[\S3]{Git-Mag},
and then carrying out an iteration of length $\kappa^{++}$ to kill all non-reflecting stationary subsets of $\kappa^+$.
However, a close inspection of Sharon's proof reveals a gap in the verification of the $\kappa^{++}$-chain-condition of the defined iteration,
and, of course, such a chain condition is crucial for the existence of a bookkeeping function
that would ensure the killing of each and every non-reflecting stationary subset of $\kappa^+$.
In a  very recent preprint \cite{bhu}, Ben-Neria, Hayut and Unger give an alternative proof of Sharon's result;
their proof does not involve iterated forcing to kill the non-reflecting stationary sets and instead uses iterated ultrapowers to avoid the generation of non-reflecting stationary sets.

In this work, we show that Sharon's original approach is repairable and, in fact, falls into our framework.
As a first step, we show that his notion of forcing for killing a single non-reflecting stationary set fits into the $\Sigma$-Prikry class:
\begin{theorem}\label{one}
Suppose $(\mathbb P_1,\lh_1,c_1)$ is $\Sigma$-Prikry and $\dot T$
is a $\mathbb P_1$-name for a non-reflecting stationary subset of $E^\mu_\omega$.
Then there exists a corresponding triple
$(\mathbb P_2,\lh_2,c_2)$ such that:
\begin{itemize}
\item $\mathbb P_2$ is a notion of forcing that projects to $\mathbb P_1$; furthermore:
\item $(\mathbb P_2,\lh_2,c_2)$ is $\Sigma$-Prikry
admitting a \emph{forking projection} to $(\mathbb P_1,\lh_1,c_1)$;
\item $\one_{\mathbb P_2}$ forces that $\dot T$ is nonstationary.
\end{itemize}
\end{theorem}

The exact definition of \emph{forking projection} may be found in Section~\ref{sectionforking},
but, roughly speaking, this is a kind of projection that ensures a much better correspondence between the two $\Sigma$-Prikry triples,
which later allows to iterate this procedure.
In a sequel to this paper \cite{partII}, we present our iteration scheme for $\Sigma$-Prikry notions of forcing,
from which we obtain a correct proof of (a strong form of) Sharon's result:
\begin{theorem}\label{thm2} Suppose that $\langle \kappa_n\mid n<\omega\rangle$ is a strictly increasing sequence of Laver-indestructible supercompact cardinals. Denote $\kappa:=\sup_{n<\omega}\kappa_n$.
Then there exists a cofinality-preserving forcing extension
in which $\kappa$ remains a strong limit, $2^\kappa=\kappa^{++}$,
and every finite collection of stationary subsets of $\kappa^+$ reflects simultaneously.
\end{theorem}

\begin{remark} The preceding is optimal as, by Corollary~\ref{prop42} below,
if $\kappa$ is an uncountable strong limit cardinal of countable cofinality,
admitting a stationary set $S\s\kappa^+$ with the property that every
countable collection of stationary subsets of $S$ reflects simultaneously, then $2^\kappa=\kappa^+$.
\end{remark}

\subsection{Notation and conventions} Our forcing convention is that $p\le q$ means that $p$ extends $q$.
We write $\cone{q}$ for $\{ p\in\mathbb P\mid p\le q\}$.
Denote $E^\mu_{\theta}:=\{\alpha<\mu\mid \cf(\alpha)=\theta\}$. The sets $E^\mu_{<\theta}$ and $E^\mu_{>\theta}$ are defined in a similar fashion.
For a stationary subset $S$ of a regular uncountable cardinal $\mu$, we write $\tr(S):=\{\delta\in E^\mu_{>\omega}\mid S\cap\delta\text{ is stationary in }\delta\}$.
$H_\nu$ denotes the collection of all sets of hereditary cardinality less than $\nu$.
For every set of ordinals $x$, we denote $\cl(x):=\{ \sup(x\cap\gamma)\mid \gamma\in\ord, x\cap\gamma\neq\emptyset\}$
and $\acc^+(x):=\{ \alpha<\sup(x)\mid \sup(x\cap\alpha)=\alpha>0\}$.
For two sets of ordinals $x,y$, we write $x\sq y$ iff there exists an ordinal $\alpha$ such that $x=y\cap\alpha$.

\section{An abstract approach to Prikry-type forcing}\label{SPS}
\begin{definition} We say that $(\mathbb P,\lh)$ is a \emph{graded poset}
iff $\mathbb P=(P,\le)$ is a poset, $\lh:P\rightarrow\omega$ is a surjection, and, for all $p\in P$:
\begin{itemize}
\item  For every $q\le p$, $\lh(q)\geq\lh(p)$;
\item  There exists $q\le p$ with $\lh(q)=\lh(p)+1$.
\end{itemize}
\end{definition}
\begin{conv} For a graded poset as above, 
we denote $P_n:=\{p\in P\mid \lh(p)=n\}$,
$P_n^p:=\{ q\in P\mid  q\le p, \lh(q)=\lh(p)+n\}$,
and sometime write $q\le^n p$ (and say the $q$ is \emph{an $n$-step extension} of $p$) rather than writing $q\in P^p_n$.
\end{conv}

\begin{definition}\label{SigmaPrikry}Suppose that $\mathbb P=(P,\le)$ is a notion of forcing with a greatest element $\one$,
and that $\Sigma=\langle \kappa_n\mid n<\omega\rangle$ is a non-decreasing sequence of regular uncountable cardinals,
converging to some cardinal $\kappa$.
Suppose that $\mu$ is a cardinal such that $\one\Vdash_{\mathbb P}\check\mu=\check\kappa^+$.\footnote{More explicitly, $\one\Vdash_{\mathbb P}\check\mu=(\check\kappa)^+$.}
For functions $\lh:P\rightarrow\omega$ and $c:P\rightarrow \mu$,
we say that $(\mathbb P,\lh,c)$ is \emph{$\Sigma$-Prikry} iff all of the following hold:
\begin{enumerate}
\item\label{c4} $(\mathbb P,\lh)$ is a graded poset;
\item\label{c2} For all $n<\omega$, $\mathbb P_n:=(P_n\cup\{\one\},\le)$ is $\kappa_n$-directed-closed;\footnote{That is, for every $D\in[P_n\cup\{\one\}]^{<\kappa_n}$ with the property that for all $p,p'\in D$,
there is $q\in D$ with $q\le p,p'$, there exists $r\in P_n$ such that $r\le p$ for all $p\in D$.}
\item\label{c1} For all $p,q\in P$, if $c(p)=c(q)$, then $P_0^p\cap P_0^q$ is non-empty;
\item\label{c5} For all $p\in P$, $n,m<\omega$ and $q\le^{n+m}p$, the set $\{r\le ^n p\mid  q\le^m r\}$ contains a greatest element which we denote by $m(p,q)$.\footnote{By convention, a greatest element, if exists, is unique.}
In the special case $m=0$, we shall write $w(p,q)$ rather than $0(p,q)$;\footnote{Note that $w(p,q)$ is the weakest extension of $p$ above $q$.}
\item\label{csize} For all $p\in P$,
the set $W(p):=\{w(p,q)\mid q\le p\}$ has size ${<}\mu$;
\item\label{itsaprojection} For all $p'\le p$ in $P$, $q\mapsto w(p,q)$ forms an order-preserving map from $W(p')$ to $W(p)$;
\item\label{c6}  Suppose that $U\s P$ is a $0$-open set, i.e., $r\in U$ iff $P^r_0\s U$.
Then, for all $p\in P$ and $n<\omega$, there is $q\le^0 p$, such that, either $P^{q}_n\cap U=\emptyset$ or $P^{q}_n\s U$.
\end{enumerate}
\end{definition}

Let us elaborate on the above definition.
\begin{itemize}
\item Here, $q$ is a ``direct extension'' of $p$ in the usual Prikry sense iff $q\le^0 p$.
Note that $q\le^0 w(p,q)\le p$. Also, it is clear that if $p\le^n q$ and $q\le^ m r$, then $p\le^{n+m} r$.
\item 
The sets $P_n^p$ consist of exactly the $n$-step extensions of $p$, and ${P}_n$ is the set of all conditions of ``length'' $n$,
i.e., the  $n$-step extensions of $\one$.
Note that, typically,  $\mathbb P_n$ is not a complete suborder of $\mathbb P$,
and that, for all $p,q\in P_n$, $p\le q$ iff $p\le^0 q$.
Thereby, $\mathbb P_n$  is not necessarily separative.

\textbf{Convention.} Whenever we talk about forcing with one of the $\mathbb P_n$'s, we actually mean that we force with its separative quotient.

\item Clause~(\ref{c1}) is a very strong form of a chain condition, stronger than that of being $\mu^+$-Knaster,
and even stronger than the notion of being $\mu^+$-2-linked. Indeed, a poset $(P,\le)$ is \emph{$\mu^+$-2-linked} iff there exists a function $c:P\rightarrow\mu$
with the property that $c(p)=c(q)$ entails that $p$ and $q$ are compatible,
whereas, here, we moreover require that such a compatibility will be witnessed by a $0$-step extension of $p$ and $q$.

\textbf{Convention.} To avoid encodings, we shall often times define the function $c$ as a map from $P$ to some natural set $\mathfrak M$ of size $\leq\mu$,
instead of a map to the cardinal $\mu$ itself.
In the special case that $\mu^{<\mu}=\mu$, we may as well take $\mathfrak M$ to be $H_{\mu}$.

\item For every $p\in P$, the set $W(p)$ is called \emph{the $p$-tree}.
For every $n<\omega$, write $W_n(p):=\{ w(p,q)\mid q\in P^p_n\}$,
and $W_{\geq n}(p):=\bigcup_{m=n}^\infty W_m(p)$.
By Lemma~\ref{lemma7} below,  $(W(p),\ge)$ is a tree of height $\omega$ whose $n^{th}$ level
is a maximal antichain in $\cone{p}$ for every $n<\omega$.

\item Clause~(\ref{c6}) is what we call the \emph{Complete Prikry Property} (CPP), an analogue of the notion of a \emph{completely Ramsey} subset of $[\omega]^\omega$.
We shall soon show (Corollary~\ref{l6} below) that it is a simultaneous generalization of the usual Prikry Property (PP) and the Strong Prikry Property (SPP).
\end{itemize}

\begin{definition} Let $d:P\rightarrow\theta$ be some coloring, with $\theta$ a nonzero cardinal.
\begin{enumerate}
\item $d$ is said to be \emph{$0$-open} iff
$d(p)\in\{0,d(q)\}$ for every pair $q\le^0 p$ of elements of $P$;
\item We say that $H\s P$ is \emph{a set of indiscernibles for $d$} iff, for all $p,q\in H$,
 $(\lh(p)=\lh(q))\implies (d(p)=d(q))$.
\end{enumerate}
\end{definition}
\begin{remark}\label{characteristic} The characteristic function $d:P\rightarrow2$ of a subset $D\s P$ is $0$-open iff $D$ is a $0$-open.
\end{remark}
\begin{lemma}\label{RamseyPrikry}
For every $p\in P$, every cardinal $\theta$ with $\log(\theta)<\kappa_{\lh(p)}$ and
every $0$-open coloring $d:P\rightarrow\theta$,\footnote{Here, $\log(\theta)$ stands for the least cardinal $\nu$ to satisfy $2^\nu\geq\theta$.}
there exists $q\le^0 p$ such that $\cone{q}$ is a set of indiscernibles for $d$.
\end{lemma}
\begin{proof} Let $p\in P$ and $d:P\rightarrow\theta$ as above.
Fix an infinite cardinal $\chi<\kappa_{\lh(p)}$ such that $2^{\chi}\geq\theta$.
Fix an injective sequence $\vec f=\langle f_\alpha\mid \alpha<\theta\rangle$ consisting of functions from $\chi$ to $2$
such that, in addition, $f_0$ is the constant function from $\chi$ to $\{0\}$.
\begin{claim} Let $i<\chi$. The set $U_i:=\{r\in P\mid f_{d(r)}(i)\neq 0\}$ is $0$-open.
\end{claim}
\begin{proof} Let $r\in U_i$ and $r'\le^0 r$.
As $r\in U_i$, $f_{d(r)}$ is not the constant function from $\chi$ to $\{0\}$,
so that $d(r)\neq 0$. Since $d$ is a $0$-open coloring, it follows that $d(r')=d(r)$.
Consequently, $r'\in U_i$, as well.
\end{proof}

Fix a bijection $e:\chi\leftrightarrow\chi\times\omega$.
We construct a $\le^0$-decreasing sequence of conditions $\langle p_\beta\mid \beta\leq\chi\rangle$ by recursion, as follows.

$\br$ Let $p_0:=p$.

$\br$ Suppose that $\beta<\chi$ and that $\langle p_\gamma\mid \gamma\leq\beta\rangle$ has already been defined.
Denote $(i,n):=e(\beta)$. Now, appeal to Definition~\ref{SigmaPrikry}(\ref{c6}) with $U_i$, $p_\beta$ and $n$
to obtain $p_{\beta+1}\le^0 p_\beta$ such that, either $P_n^{p_{\beta+1}}\cap U_i=\emptyset$ or $P_n^{p_{\beta+1}}\s U_i$.

$\br$ For every limit nonzero $\beta\leq\chi$ such that $\langle p_\gamma\mid \gamma<\beta\rangle$ has already been defined,
appeal to Definition~\ref{SigmaPrikry}(\ref{c2}) to find a lower bound $p_\beta$ for the sequence.

At the end of the above recursion, let us put $q:=p_{\chi}$, so that $q\le^0 p$.
We claim that $\cone{q}$ is a set of indiscernibles for $d$.

Suppose not, and pick two extensions $r,r'$ of $q$ such that $\lh(r)=\lh(r')$ but $d(r)\neq d(r')$.
As $d(r)\neq d(r')$ and $\vec f$ is injective, let us fix $i<\chi$ such that $f_{d(r)}(i)\neq f_{d(r')}(i)$.
Consequently, $|\{r,r'\}\cap U_i|=1$.
Now, put $n:=\lh(r)-\lh(p)$, so that $r,r'\in P^q_n$. Set $\beta:=e^{-1}(i,n)$.
By the choice of $p_{\beta+1}$, then, either $P_n^{p_{\beta+1}}\cap U_i=\emptyset$ or $P_n^{p_{\beta+1}}\s U_i$.
As $q\le^0 p_{\beta+1}$, we have $\{r,r'\}\s P_n^{p_{\beta+1}}$,
contradicting the fact that $|\{r,r'\}\cap U_i|=1$.
\end{proof}

It follows that the Complete Prikry Property (CPP) implies the Prikry property (PP) as well as the Strong Prikry property (SPP).

\begin{cor}\label{l6} Let $p\in P$.
\begin{enumerate}
\item Suppose $\varphi$ is a sentence in the forcing language.
Then there is $q\le^0 p$ that decides $\varphi$;
\item  Suppose $D\s P$ is a $0$-open set which is dense below $p$.
Then there are $q\le^0 p$ and $n<\omega$ such that $P^q_n\s D$.\footnote{Note that if $D$ is open, then, moreover, $P^q_m\s D$ for all $m\geq n$.}
\end{enumerate}
\end{cor}
\begin{proof} (1)
Define a $0$-open coloring $d:P\rightarrow3$, by letting, for all $r\in P$,
$$d(r):=\begin{cases}
2,&\text{if }r\Vdash\neg\varphi;\\
1,&\text{if }r\Vdash\varphi;\\
0,&\text{otherwise}.
\end{cases}$$

Appeal to Lemma~\ref{RamseyPrikry} with $d$ to get a corresponding $q\le^0p$.
Towards a contradiction, suppose that $q$ does not decide $\varphi$.
In other words, there exist $q_1\le q$ and $q_2\le q$ such that $d(q_1)=1$ and $d(q_2)=2$.
By possibly iterating Clause~(\ref{c4}) of Definition~\ref{SigmaPrikry} finitely many times,
we may find $r_1\le q_1$ and $r_2\le q_2$ such that $\lh(r_1)=\lh(r_2)$.
By definition of $d$, we have $d(r_1)=1$ and $d(r_2)=2$.
Finally, as $r_1$ and $r_2$ are two extensions $q$ of the same ``length'', $1=d(r_1)=d(r_2)=2$. This is a contradiction.

(2) Define a coloring $d:P\rightarrow2$ via $d(r):=1$ iff $r\in D$.
By Remark~\ref{characteristic}, we may
appeal to Lemma~\ref{RamseyPrikry} with $d$ to get a corresponding $q\le^0 p$.
As $D$ is dense, let us fix $r\in D$ extending $q$. Let $n:=\lh(r)-\lh(p)$,
so that $d\restriction P^q_n$ is constant with value $d(r)$.
Recalling that $r\in D$ and the definition of $d$, we infer that $P^q_n\s D$.
\end{proof}

\begin{lemma}[The $p$-tree]\label{lemma7} Let $p\in P$.
\begin{enumerate}
\item For every $n<\omega$, $W_n(p)$ is a maximal antichain in $\cone{p}$;
\item Every two compatible elements of $W(p)$ are comparable;
\item For any pair $q'\le q$ in $W(p)$,  $q'\in W(q)$;
\item $c\restriction W(p)$ is injective.
\end{enumerate}
\end{lemma}
\begin{proof} 
(1) Clearly, $W_0(p)=\{p\}$ is a maximal antichain below $p$. Thus, hereafter, assume that $n>0$.

$\br$ To see that $W_n(p)=\{w(p,q)\mid q\in P^p_n\}$ is an antichain, suppose that $q_1,q_2\in P^p_n$ are such that $w(p,q_1)$ and $w(p,q_2)$ are compatible, as witnessed by some $q$.
By Definition~\ref{SigmaPrikry}\eqref{c4}, $q\in P^p_{n+m}$ for some $m<\omega$.
By Definition~\ref{SigmaPrikry}\eqref{c5}, then, $\{r\in P^p_{n}\mid  q\le r\}$ contains a greatest element, say, $r^*$.
Let $i<2$ be arbitrary. As $q\le w(p,q_i)$, it is not hard to see that $w(p,q_i)$ is the greatest element in $\{r\in P^p_{n}\mid  q\le r\}$, so that $w(p,q_i)=r^*$.
Altogether, $w(p,q_1)=r^*=w(p,q_2)$.

$\br$ To verify maximality of the antichain $W_n(p)$, let $p'\le p$ be arbitrary.
By Definition~\ref{SigmaPrikry}\eqref{c4}, let us pick some $q\in P^{p'}_n$, so that $q\in P^{p}_{n+m}$ for some $m<\omega$.
Then, by Definition~\ref{SigmaPrikry}\eqref{c5}, $\{r\in P^p_{n}\mid  q\le r\}$ contains a greatest element, say, $r^*$.
As $w(p,r^*)=r^*$, we have $r^*\in W_n(p)$. In addition, $r^*$ and $p'$ are compatible, as witnessed by $q$.

(2) Suppose that $q_0, q_1\in W(p)$ are two compatible elements.
Fix integers $n_0,n_1$ such that $q_0\in W_{n_0}(p)$ and $q_1\in W_{n_1}(p)$. 

If $n_0=n_1$, then by Clause~(1), $q_0=q_1$. Thus, without loss of generality, assume that $n_0<n_1$.
Let $r^*$ be the greatest element of $\{r\in P^p_{n_0}\mid  q_1\le r\}$. 
Then $r^* = w(p, r^*)\in W_{n_0}(p)$ and $q_1$ witnesses that $r^*$ is compatible with $q_0$. 
So $r^*$ and $q_0$ are compatible elements of $W_{n_0}(p)$,
and hence $q_1\le r^*= q_0$.

(3) Given $q'\le q$ as above, let $r'\in P^p$ be such that  $q'=w(p,r')$. 
Now, to prove that $w(p,r')\in W(q)$, it suffices to show that $w(p,r')=w(q,r')$. Here goes:

$\br$ As $r'\le w(q,r')\le q\le p$, we infer that $w(q,r')\in \{ s\mid r'\le s\le p\}$,
so that $w(q,r')\le w(p,r')$.

$\br$ As $r'\le w(p,r')=q'\le q$, we infer that $w(p,r')\in \{ s\mid r'\le s\le q\}$,
so that $w(p,r')\le w(q,r')$. 

(4)  By Definition~\ref{SigmaPrikry}(\ref{c1}), for all $q,q'\in W(p)$, if $c(q)=c(q')$, then $q$ and $q'$ are compatible, and they have the same $\lh$-value.
It now follows from Clause~(1) that $c\restriction W(p)$ is injective.
\end{proof}

\begin{lemma}\label{l15} Suppose that $\bar p\le p'\le p$ and $q\in W(\bar p)$. Then $w(p,q)=w(p,w(p',q))$.\footnote{For future reference, we point out that this fact relies only on clauses (\ref{c4}) and (\ref{c5}) of Definition~\ref{SigmaPrikry}.}
\end{lemma}
\begin{proof}
As $\lh(w(p,q))=\lh(q)=\lh(w(p',q))=\lh(w(p,w(p',q))$,
we infer the existence of some $n<\omega$ such that both $w(p,q)$ and $w(p,w(p',q))$ belong to $W_n(p)$.
By Lemma~\ref{lemma7}(1), then, it suffices to verify that the two are compatible.
And indeed, we have $q\le w(p,q)$ and $q\le w(p',q)\le w(p,w(p',q))$.
\end{proof}

\begin{lemma}\label{l14} 
\begin{enumerate}
\item\label{C1l14} $\mathbb P$ does not add bounded subsets of $\kappa$;
\item\label{C2l14} For every regular cardinal $\nu\geq\kappa$, if there exists $p\in P$ for which $p\Vdash_{\mathbb P}\cf(\check\nu)<\check\kappa$,
then there exists $p'\le p$ with $|W(p')|\geq\nu$;\footnote{For future reference, we point out that this fact relies only on clauses (\ref{c4}),(\ref{c2}),(\ref{c5}) and (\ref{c6}) of Definition~\ref{SigmaPrikry}.
Furthermore, we do not need to know that $\one$ decides a value for $\kappa^+$.}
\item\label{C3l14} Suppose $\one\Vdash_{\mathbb P}``\check\kappa\text{ is singular}"$. Then $\mu=\kappa^+$ iff, for all $p\in P$, $|W(p)|\leq\kappa$.
\end{enumerate}
\end{lemma}
\begin{proof} (1) Suppose that $p$ forces that $\sigma$ is a name for a subset of some $\theta<\kappa$.
By possibly iterating Clause~(\ref{c4}) of Definition~\ref{SigmaPrikry} finitely many times, we may find $p'\le p$ with $\kappa_{\lh(p')}>\theta$. Denote $n:=\lh(p')$.
Then by Corollary~\ref{l6}(1) and Definition~\ref{SigmaPrikry}\eqref{c2},
we may find a $\le_0$-decreasing sequence of conditions, $\langle p_\alpha\mid \alpha\leq\theta\rangle$,
with $p_0\le^0 p'$, such that, for each $\alpha<\theta$,  $p_\alpha$ $\mathbb P$-decides whether $\alpha$ belongs to $\sigma$.
Then $p_\theta$ forces that $\sigma$ is a ground model set.

(2) Suppose $\theta,\nu$ are regular cardinals with $\theta<\kappa\leq\nu$, $\dot f$ is a $\mathbb P$-name for a function from $\theta$ to $\nu$,
and $p\in P$ is a condition forcing that the image of $f$ is cofinal in $\nu$.
Denote $n:=\lh(p)$.
By Definition~\ref{SigmaPrikry}\eqref{c4}, we may assume that $\kappa_{n}>\theta$.
For all $\alpha<\theta$, let $D_\alpha$ denote the open set of conditions below $p$ that $\mathbb P$-decides a value for $f(\alpha)$.
As $D_\alpha$ is dense below $p$, by Corollary~\ref{l6}(2) and Definition~\ref{SigmaPrikry}\eqref{c2},
we may find a $\le_0$-decreasing sequence of conditions $\langle p_\alpha\mid \alpha<\theta\rangle$, with $p_0\le^0 p$,
and a sequence $\langle n_\alpha\mid \alpha<\theta\rangle$ of elements of $\omega$,
such that, for all $\alpha<\theta$, $P^{p_\alpha}_{n_\alpha}\s D_\alpha$.

By Definition~\ref{SigmaPrikry}\eqref{c2}, let $p'$ be a lower bound for $\{p_\alpha\mid \alpha<\theta\}$.
Evidently, $P^{p'}_{n_\alpha}\s D_\alpha$ for every $\alpha<\theta$.
Now, let $$A_\alpha:=\{\beta<\nu\mid \exists p\in P^{p'}_{n_\alpha}[p\Vdash_{\mathbb P}\dot{f}(\check{\alpha})=\check{\beta}]\}.$$
By Lemma~\ref{lemma7}(1), we have $A_\alpha=\{\beta<\nu\mid \exists p\in W_{n_\alpha}(p')[p\Vdash_{\mathbb P}\dot{f}(\check{\alpha})=\check{\beta}]\}$.
Let $A:=\bigcup_{\alpha<\theta}A_\alpha$.
As $|A|\leq\sum_{\alpha<\theta}|W_{n_\alpha}(p')|\leq\theta\cdot |W(p')|$,
it follows that if $|W(p')|<\nu$, then $\sup(A)<\nu$,
and $p'$ forces that the range of $f$ is bounded below $\nu$, which would form a contradiction.
So $|W(p')|\geq\nu$.

(3) The forward implication follows from Definition~\ref{SigmaPrikry}\eqref{csize}.

Next, suppose that, for all $p\in P$, $|W(p)|\leq\kappa$.
Towards a contradiction, suppose that there exist $p\in P$ forcing that $\kappa^+$ is collapsed.
Denote $\nu:=\kappa^+$.
As $\one\Vdash_{\mathbb P}``\check\kappa\text{ is singular}"$, this means that $p\Vdash_{\mathbb P}\cf(\check\nu)<\check\kappa$,
contradicting Clause~(2).
\end{proof}

\section{Examples}\label{examples}

\subsection{Vanilla Prikry} Throughout this subsection assume that $\kappa$ is a measu\-rable cardinal and that $\mathcal{U}$ is a normal measure over it.
We shall show that the classical Prikry forcing $\mathbb{P}$ to singularize $\kappa$ to cofinality $\omega$ fits into the $\Sigma$-Prikry framework. Recall that $\mathbb{P}:=(P,\le)$, where
conditions in $P$ are pairs of the form $p=(s,A)$, with $s$ being a finite increasing sequence in $\kappa$ and $A\in \mathcal{U}$ with $\sup(s)<\min(A)$.
The ordering $\le$ is defined by $(s,A)\le (t,B)$ iff $t\sqsubseteq s$, $A\subseteq B$ and $s\setminus t \subseteq B$.

Let $X\in [{}^{<\omega}\kappa]^\kappa$. The diagonal intersection of a family $\{ A_s\mid s\in X\}\s\mathcal{U}$ is given by
$$\diagonal \{ A_s\mid s\in X\}:=\{\alpha<\kappa\mid \forall s\in X (\max(s)<\alpha\rightarrow \alpha\in A_s)\}.$$
Since $\mathcal{U}$ is normal, $\diagonal \{ A_s\mid s\in X\}\in \mathcal{U}$.

Let $\Sigma$ be the $\omega$-sequence with constant value $\kappa$ and $\mu:=\kappa^+$. The notion of length associated to $\mathbb{P}$, $\lh: P\rightarrow \omega$, is given by $\lh(s,A):=|s|$.
Finally, define $c:P\rightarrow{}^{<\omega}\kappa$ via $c(s,A):=s$.
In the next proposition we verify that $(\mathbb{P},\lh, c)$ is $\Sigma$-Prikry.
\begin{prop}\label{VanillaPrikryIsSigmaPriky}
$(\mathbb{P},\lh, c)$ is $\Sigma$-Prikry.
\end{prop}
\begin{proof}
We go over the clauses of Definition~\ref{SigmaPrikry}.
\begin{enumerate}
\item For $p=(s,A)\in P$, $(s{}^\smallfrown\langle \nu\rangle, A\setminus \nu+1)\in P^p_1$, for all $\nu\in A$. Moreover, by definition of $\le$, if $q\le p$ then $\lh(q)\geq \lh(p)$.
\item Follows from the $\kappa$-completeness of $\mathcal{U}$.
\item Let $p,q\in P$ and assume that $c(p)=c(q)=s$. Set $p:=(s,A)$ and $q:=(s,B)$. Clearly $(s, A\cap B)$ is in $P^{p}_0\cap P^q_0$.
\item Let $p:=(s,A)\in P$, $n,m<\omega$ and $q:=(t,B)\in P^p_{n+m}$. Set $u:=t\upharpoonright (|s|+n)$. Then $r^*:=(u, A\setminus \max(u)+1)$ is the greatest element in $\{r\in P^p_n\mid q\le r\}$.

\item Let $p\in P$ and $n<\omega$. Denoting $p:=(s,A)$, we have that
$W_n(p)=\{(s{}^\smallfrown t, A\setminus \max(t)+1)\mid t\in [A]^{n}, t\text{ is increasing }\}$.
Clearly, $|W_n(p)|=\kappa<\mu$.

\item Let $p'\le p$ and $q,q'\in W(p')$ and assume $q'\le q$. Set $p:=(s,A)$,  $q:=(t,B)$ and $q':=(u,C)$. By the previous items, $w(p,q)=(t, A\setminus \max(t)+1)$ and $w(p,q')=(u,A\setminus \max(u)+1)$ and, since $q'\le q$, is clear that $w(p,q')\le w(p,q)$, as desired.

\item\label{CCPVanilla} This follows in a similar fashion to the classical proof of the SPP in \cite[Lemma 1.13]{Gitik-handbook}.\qedhere
\end{enumerate}
\end{proof}

As a corollary, we infer that the product of two $\Sigma$-Prikry notions of forcing need not be $\Sigma$-Prikry.
For this, let  $\mathcal{U}$ and $\mathcal{V}$ be  normal measures over the same measurable cardinal $\kappa$ and let $\mathbb P$ and $\mathbb Q$ be the corresponding Vanilla Prikry notions of forcing.
We claim that $\mathbb P\times\mathbb Q$ adds a bounded subset of $\kappa$, so that, by Lemma~\ref{l14}\eqref{C1l14}, it is not $\Sigma$-Prikry.

Let $\vec{s}=\langle s_n\mid n<\omega\rangle$ and $\vec{t}=\langle t_n\mid n<\omega\rangle$ be pairwise generic Prikry-sequences with respect to $\mathbb P$ and $\mathbb Q$, i.e., $\vec{s}$ (resp.~$\vec{t}$) generates a generic filter for $\mathbb P$ (resp.~$\mathbb Q$) and furthermore $\vec{s}\notin V[\vec{t}]$  and $\vec{t}\notin V[\vec{s}]$. By mutual genericity, $X:=\{n\in\omega\mid s_n< t_n\}$ is infinite and it is also not hard to check that $X\notin V$.
In particular, $\mathbb{P}\times \mathbb Q$ adds a real.

\subsection{Supercompact Prikry forcing}

Let $\kappa<\lambda$ be two cardinals and assume that  is $\mathcal{U}$ a $\lambda$-supercompact measure on $\mathcal{P}_\kappa(\lambda)$, namely, $\mathcal{U}$ is a $\kappa$-complete, normal and fine ultrafilter over $\mathcal{P}_\kappa(\lambda)$ (cf.~\cite[p. 301]{Kan}). In this section we prove that $\mathbb{P}$, the Supercompact Prikry forcing with respect to $\mathcal{U}$ for  singularizing $\kappa$ to cofinality $\omega$ and collapse the interval $[\kappa, \lambda^{<\kappa}]$, falls also into the $\Sigma$-Prikry framework. Recall that for $x,y\in \mathcal{P}_\kappa(\lambda)$, $x\prec y$ iff $x\s y$ and $\otp(x)<\otp(y\cap \kappa)$.

Recall that conditions are of the form $(\vec{x}, A)$, where $\vec{x}$ is a finite $\prec$-increasing sequence in $\mathcal{P}_\kappa(\lambda)$, called the {\it stem} of the condition,  and $A\in \mathcal{U}$.  $(\vec{x},A)\le (\vec{y},B)$ iff $\vec{y}\sqsubseteq \vec{x}$, $\vec{x}\setminus \vec{y}\subseteq B$ and $A\subseteq B$.

Given a set of stems $X$ the diagonal intersection of a family $\{A_s\mid s\in X\}\s \mathcal{U}$ is given by
$$\diagonal \{A_s \mid s\in X\}:=\{y\in \mathcal{P}_\kappa(\lambda)\mid \forall s\in X(s\prec y\rightarrow y\in A_s)\}.$$
Again, normality of $\mathcal{U}$ implies that $\diagonal \{A_s \mid s\in X\}\in \mathcal{U}$.
Also, one can prove a version of the classical R\"owbottom Lemma for $\lambda$-supercompact measures.

Let $\Sigma$ be the $\omega$-sequence with constant value $\kappa$ and $\mu:=(\lambda^{<\kappa})^+$.
The notion of length associated to $\mathbb{P}$, $\lh: P\rightarrow \omega$, is given by $\lh(\vec{x},A):=|\vec{x}|$.

Finally, define $c:P\rightarrow{}^{<\omega}(\lambda^{<\kappa})$ via $c(\vec x,A):=\vec x$.
Mimicking the proof of Proposition \ref{VanillaPrikryIsSigmaPriky} one can prove the next proposition:

\begin{prop}\label{SuperCompactPrikryIsSigmaPrikry}
$(\mathbb{P},\lh, c)$ is $\Sigma$-Prikry.\qed
\end{prop}

\subsection{Diagonal Supercompact Prikry Forcing}
Here we show that the Diagonal Supercompact Prikry Forcing, due to Gitik and Sharon \cite{GitSha}, can be regarded as a $\Sigma$-Prikry forcing. For economy of the discourse henceforth we shall refer to the Diagonal Supercompact Prikry Forcing simply as \textsf{GS} forcing, where the abbreviation \textsf{GS} stands for Gitik-Sharon.

Let $\langle\kappa_n\mid n<\omega\rangle$ be an increasing sequence of regular cardinals, and denote $\kappa:=\kappa_0$. Let $\Sigma$ be the $\omega$-sequence with constant value $\kappa$ and $\mu:=(\sup_n\kappa_n)^+$. Suppose that $\mathcal{U}$ is a supercompact measure on $\mathcal{P}_\kappa(\mu^+)$,
and let $U_n$ be its projection projection onto $\mathcal{P}_\kappa(\kappa_n)$.\footnote{Namely, for each $X\s \mathcal{P}_\kappa(\kappa_n)$, $X\in U_n$ iff $\pi_n^{-1}[X]\in \mathcal{U}$,
where $\pi_n$ is the standard projection between
$\mathcal{P}_\kappa(\mu^+)$ and $\mathcal{P}_\kappa(\kappa_n)$.}
It is routine to check that, for each $n<\omega$, $U_n$ is a $\kappa_n$-supercompact measure over $\mathcal{P}_\kappa(\kappa_n)$.

We begin defining the universe $P$ of the \textsf{GS} poset $\mathbb{P}$:
\begin{definition}\label{GSUniverse}
Define $P$ as the set of sequences $p=\langle x^p_0,\dots, x^p_{n-1} A^p_n,\allowbreak A^p_{n+1},\dots \rangle$ such that each $x_i\in \mathcal{P}_\kappa(\kappa_i)$, $x_i\prec x_{i+1}$, and $A_k\in U_k$.
Denote $\lh(p):=n$ and call the sequence $\langle x_0,\dots, x_{n-1}\rangle$ the \emph{stem of $p$}. Typically we will denote this sequence by $\stem(p)$. The order is the usual: we extend the stems by picking elements from the measure one sets, and then shrink the measure one sets.
\end{definition}

\begin{definition}
Let $p=\langle x^p_0,\dots, x^p_{n-1}, A^p_n, A^p_{n+1}, \dots, \rangle$ in $P$.
For  $x\in A^p_{\lh(p)}$,  $p{}^\curvearrowright\langle x\rangle$ stands for the unique condition
$$q:=
\langle x^p_0,\dots, x^p_{\lh(p)-1}, x, B^p_{\lh(p)+1}, B^p_{\lh(p)+2}, \dots\rangle,
$$
where, for each $i\geq \lh(p)$, $B^p_i:=\{y\in A^p_i\mid x\prec y \}$.
Similarly, for all $n\geq \lh(p)$, and any $\prec$-increasing $\vec{x}:=\langle x_{\lh(p)},\dots, x_{n+1}\rangle\in \prod_{i=\lh(p)}^{n+1} A^p_i$, we define $p{}^\curvearrowright\vec{x}$ to be the weakest extension of $p$ with stem equal to $\stem(p){}^\smallfrown \vec{x}$.

\end{definition}
Note that whenever $q\le p$, for some $\vec{x}$, we have that $q\le ^0 p{}^\curvearrowright\vec{x}\le p$. I.e. this is exactly the needed notion to verify clauses \eqref{c5}, \eqref{csize}, \eqref{itsaprojection} of Definition~\ref{SigmaPrikry}. In particular, for $q,p$
as above, $w(p,q)=p{}^\curvearrowright\vec{x}$.

Finally, define $c:P\rightarrow {}^{<\omega}(P_\kappa(\kappa^{+\omega}))$ via
$$c(\langle x^p_0, \dots, x^p_{\lh(p)-1}, A^p_{\lh(p)},A^p_{\lh(p)+1},\dots\rangle):=\langle x^p_0, \dots, x^p_{\lh(p)-1}\rangle.$$

\begin{prop}\label{GSSigmaPriky}
$(\mathbb{P},\lh, c)$ is $\Sigma$-Prikry.
\end{prop}

\begin{proof}
We go over the clauses of Definition~\ref{SigmaPrikry}.

Clause~\eqref{c2} follows from the completeness of the normal measures. Clauses\linebreak~\eqref{c4} and \eqref{c1} are clear. Clauses \eqref{c5}, \eqref{csize} follow from the above discussion. In particular for any $p$, $W_n(p)=\{p{}^\curvearrowright\vec{x}\mid \vec{x}\in \prod_{i=\lh(p)}^{n-1} A^p_i, \prec\text{-increasing}\}$, which has cardinality $\kappa_n$. Clause~\eqref{itsaprojection} follows from the definition of the ordering. And Clause~\eqref{c6} follows in a similar fashion to the proof of the SPP for the \textsf{GS} poset.
\end{proof}

\subsection{AIM forcing}
We now consider the notion of forcing from \cite{AIM}.
 Suppose $\mu$ is a strongly inaccessible cardinal,
 and  $\Sigma=\langle \kappa_n\mid n<\omega\rangle$ is a strictly increasing sequence of $\mu$-supercompact cardinals.
 Denote $\kappa := \sup_{n<\omega}\kappa_n$.
 For each $n<\omega$, let $U_n$ be some $\kappa_n$-complete fine normal ultrafilter on $P_{\kappa_n}(\mu)$, and for $\kappa\leq\alpha<\mu$ let $U_{n, \alpha}$ be the projection of $U_n$ to
   $P_{\kappa_n}(\alpha)$ via the map $x \mapsto x \cap \alpha$.

\begin{definition} We define $(\mathbb P,\lh,c)$ with $\mathbb P=(P,\le)$, as follows.
$P$ consists of all sequences $p=\langle p_n\mid n < \omega\rangle$ such that
  for  some  $\lh(p)<\omega$, we have:
\begin{enumerate}
\item\label{domandrgef}  For each $n<\lh(p)$, $p_n$ is a function  $f^p_n$ with $\dom(f^p_n) \subseteq [\kappa, \mu)$,
    $|\dom(f^p_n)|< \mu$, and for all $\eta \in \dom(f^p_n)$, $f^p_n(\eta) \in P_{\kappa_n}(\eta)$;

\item For each $n\geq\lh(p)$, $p_n$ is a triple $(a^p_n, A^p_n, f^p_n)$,  where:
\begin{enumerate}[label={\alph*}),ref={\alph*}]

\item \label{fanda} $a^p_n$ is a subset of $[\kappa, \mu)$ with  $|a^p_n| < \mu$ that moreover admits a maximal element $\alpha^p_n$;

\item $A^p_n \in U_{n, \alpha^p_n}$;

\item  $f^p_n$ is a function with $\dom(f^p_n) \subseteq [\kappa, \mu)\setminus a_n^p$,
    $|\dom(f^p_n)|< \mu$ such that, for all $\eta \in \dom(f^p_n)$, $f^p_n(\eta) \in P_{\kappa_n}(\eta)$.
\end{enumerate}
\item \label{smallainc} $\langle a^p_n\mid \lh(p)\leq n < \omega\rangle$ is $\subseteq$-increasing.
  \end{enumerate}

We let $p \le q$ if and only if:
\begin{enumerate}

\item \label{whandlh}  $\lh(p)\geq\lh(q)$.

\item \label{fextend} For all $n$, $f^p_n \supseteq f^q_n$;

\item \label{chooseandproject} For $n$ with $\lh(q) \leq n < \lh(p)$,  $a^q_n \subseteq \dom(f^p_n)$,
   $f^p_n(\alpha^q_n) \in A^q_n$, and $f^p_n(\eta) = f^p_n(\alpha^q_n) \cap \eta$ for
   all $\eta \in a^q_n$.\footnote{This is the corresponding analogous of condition (2d) in \cite[Definition 2.10]{Gitik-handbook} for the Extender-based Prikry forcing. See also Subsection~\ref{ebpf-section} below.}

\item \label{increasing} $(f^p_n(\alpha^q_n))_ {\lh(q) \leq n < \lh(p)}$ is $\subseteq$-increasing.

\item \label{coherence} For $n\geq\lh(p)$, we have $a^q_n \subseteq a^p_n$,
 and $x \cap \alpha^q_n \in A^q_n$ for all $x \in A^p_n$.

\item \label{funky2}       For $n\geq\lh(p)$,  if $\lh(q) < \lh(p)$, then $f^p_{\lh(p)-1}(\alpha^q_{\lh(p)-1})\subseteq x$ for all $x \in A^p_n$.
\end{enumerate}

Finally, by cardinality considerations, we find $c:P\rightarrow\mu$ which is an injection.
\end{definition}

By virtue of Lemma~4 and Corollary~1 of \cite{AIM}, $\mathbb P$ collapses all cardinals $\theta$ with $\kappa < \theta < \mu$ and makes $\mu$ the successor of $\kappa$.
Next, we briefly go over the clauses of Definition~\ref{SigmaPrikry} to explain why $(\mathbb{P}, \lh, c)$ is $\Sigma$-Prikry.

By the completeness of the measures, we get that for each $n$, $\mathbb P_n$ is $\kappa_n$-directed-closed giving Clause~\eqref{c2}. Clauses~\eqref{c4} and \eqref{c1} are clear. For Clauses~\eqref{c5}, \eqref{csize}, \eqref{itsaprojection} we need to recall some definitions and facts from \cite{AIM}.

\begin{definition} For conditions $r \le q$, we let
   $\stem(r, q)$ denote the finite sequence $(f^r_i(\alpha^q_i))_{\lh(q) \leq i < \lh(r)}$.
\end{definition}

\begin{definition}\label{def2} Let $q$ be a condition. Let $l\in(\lh(q),\omega)$ and  $s \in \prod_{\lh(q) \le i < l} A^q_i$
   be a $\subseteq$-increasing sequence.
    Define $q + s$ as the $\omega$-sequence $(r_k)_{k < \omega}$ such that:
\begin{itemize}
\item   For $k < \lh(q)$,   $r_k = f^q_k$.
\item   For $\lh(q) \leq k < l$, $r_k$ is the function with
   domain $\dom(f^q_k) \cup a^q_k$ such that $r_k(\eta)=f^q_k(\eta)$ for $\eta\in\dom(f^q_k)$
   and $r_k(\eta) = s_k \cap \eta$ for $\eta \in a^q_k$.
\item   For $k\geq l$,
     $r_k = (f^q_k, a^q_k, B_k)$ where
    $B_k = \{ x \in A^q_k : s_{l-1} \subseteq x \}$.\footnote{Notice that $f^q_{l-1}(\alpha^q_{l-1})=s_{l-1}$, as $s_{l-1}\s \alpha^q_{l-1}$.}
\end{itemize}
By convention we also define $q + \langle\rangle = q$.
\end{definition}

In \cite[Lemma 8]{AIM}, it is shown that for  $q$ and $s$ as in Definition~\ref{def2}, $q + s$ is a condition in $P$ extending $q$. Moreover, for each $r\le q$, $r \le^0 q + \stem(r, q)$ and also is not hard to check that $q + \stem(r, q)$ is the weakest extension of $q$ above $r$; i.e., in our notation, $q +\stem(r,q) = w(q,r)$. Thereby, for each $n$, $W_n(q)$ is the set of all conditions of the form $q + s$, where $s\in \prod_{\lh(q)\leq i<n} A^q_i$. It thus follows that $W_n(q)$ has cardinality less than $\mu$, hence yielding clauses \eqref{c5} and \eqref{csize}.

For Clause~\eqref{itsaprojection}, let $q'\le q$ and $r_0,r_1\in W(q')$ with $r_0\le r_1$. By the previous discussion, for each $i\in 2$,  there is $s_i$ such that $r_i=q+s_i$ and $w(q,q+s_i)=q+s_i$.
Altogether, we have shown that $w(q,q+s_0)\le w(q,q+s_1)$, hence yielding Clause~\eqref{itsaprojection}.

Finally, Clause~\eqref{c6} of Definition~\ref{SigmaPrikry} follows in a similar fashion to the Prikry property arguments in \cite[Lemma 10 and 11]{AIM}.
The main point is that given a $0$-open set $U$ and a condition $p$, for every possible $s$ as in the above definitions, we  check if there is $q\le p+s$ in $U$. If there is, call it $p_s$; otherwise, let $p_s:=p+s$.
Doing this via a careful induction one constructs $q\le^0 p$, such that, for all $s$, $q+s\le^0 p_s$. Then we shrink the measure one sets to ensure that either each $q+s$ is in $U$ or none is.

To sum up, we have the following:
\begin{prop} $(\mathbb{P},\lh, c)$ is $\Sigma$-Prikry. \qed
\end{prop}

\subsection{Extender-based Prikry Forcing}\label{ebpf-section}

Suppose that $\langle\kappa_n\mid n<\omega\rangle$ is an increasing sequence of regular cardinals, let $\kappa:=\sup_{n<\omega}\kappa_n$, $\mu:=\kappa^+$ and let $\lambda>\mu$ be such that $\lambda^{<\lambda}=\lambda$. Suppose further that each $\kappa_n$ carries a $(\kappa_n, \lambda+1)$-extender $E_n:=\langle E_{n,\alpha}\mid \alpha<\lambda\rangle$. Then extender-based Prikry forcing with respect to these extenders, denoted by $\mathbb{P}_{ebpf}$, adds sequences $\langle f_n\mid n<\omega\rangle$, where each $f_n:\lambda\rightarrow\kappa_n$ is generic for the Cohen forcing $\add(\mu, \lambda)$, and an unbounded set $F\subset\lambda$ with the following properties:
\begin{itemize}
\item
setting $t_\alpha\in\prod_n\kappa_n$ by $t_\alpha(n)=f_n(\alpha)$, we have that $t_\alpha\notin V$ iff $\alpha\in F$;
\item
for all $\alpha<\beta$ both in $F$, for all large $n$,  $t_\alpha(n)<t_\beta(n)$;
\item
for all $\alpha\in F$, $t_\alpha$ is a Prikry generic sequence with respect to the measures $\langle E_{n,\alpha}\mid n<\omega\rangle$ (i.e. for all measure one sets from these ultrafilters, the sequence meets them on a tail end)

\end{itemize}
In particular, forcing with $\mathbb{P}_{ebpf}$ makes $2^{\kappa}=\lambda$. This forcing plays an important role in the proof of Theorem~\ref{thm2}.
In a sequel to this paper \cite{partII}, we will describe this forcing in detail and prove that it is $\Sigma$-Prikry, where $\Sigma:=\langle\kappa_n\mid n<\omega\rangle$.

\subsection{Lottery sum} Suppose that $\Sigma=\langle \kappa_n\mid n<\omega\rangle$ is non-decreasing sequence of regular uncountable cardinals,
converging to some cardinal $\kappa$, $\mu$ is a cardinal, and $\langle (\mathbb Q_i,\lh_i,c_i)\mid i<\nu\rangle$
is a sequence of $\Sigma$-Prikry notions of forcing such that $\nu<\mu$ and, for all $i<\nu$, $\one_{\mathbb Q_i}\Vdash_{\mathbb Q_i}\check\mu=\check\kappa^+$.

Define $P:=\{ (i,p)\mid i<\nu, p\in Q_i\}\cup\{\emptyset\}$
and an ordering $\le$, letting $(i,p)\le (j,q)$ iff $i=j$ and $p\le_{\mathbb Q_i}q$,
as well as setting $\emptyset\le x$ for any $x\in P$.  Set $\mathbb{P}:=(P,\le)$ and note that $\one_{\mathbb P}=\emptyset$ and
$\one_{\mathbb P}\Vdash_{\mathbb Q_i}\check\mu=\check\kappa^+$.
Now, define $\lh:P\rightarrow\omega$ by letting $\lh(\emptyset):=0$
and $\lh(i,p):=\lh_i(p)$. Finally, define $c:P\rightarrow\mu\times\mu$ by letting $c(\emptyset):=(0,0)$ and $c(i,p):=(i,c_i(p))$.
\begin{prop}
$(\mathbb{P},\lh, c)$ is $\Sigma$-Prikry.
\end{prop}
\begin{proof}
We go over the clauses of Definition~\ref{SigmaPrikry}.
\begin{enumerate}
\item As $(i,q)\le (j,p)$ entails $i=j$ and $q\le_{\mathbb Q_i} p$, we infer from the fact that $(\mathbb{Q}_i,\lh_i, c_i)$ is $\Sigma$-Prikry, that $\lh(i,p)=\lh(p)\leq \lh(q)=\lh(i,q)$.
\item Let $D\in [P_n\cup\{\emptyset\}]^{<\kappa_n}$ be directed. Find $i<\nu$ such that $D\setminus\{\emptyset\}\s \{i\}\times (Q_i)_n$. Now, as $(\mathbb{Q}_i,\lh_i, c_i)$ is $\Sigma$-Prikry,
there exists a lower bound $p$ for $\{ q\in (Q_i)_n\mid (i,q)\in D\}$. Evidently, $(i,p)$ is a lower bound for $D$.
\item Follows from the fact that, for all $i<\nu$, $(\mathbb{Q}_i,\lh_i, c_i)$ being $\Sigma$-Prikry.
\item[(4)-(5)] Let $x\in P$ and $(i,q)\in P^x$. If $x=\emptyset$ it is not hard to check that $w(\emptyset,\emptyset)=\emptyset$ and that,
more generally, $m(\emptyset, (i,q))=(i, m(\one_{\mathbb{Q}_i}, q))$. Hence, $W(\emptyset)\s\{\emptyset\}\,\cup\, \bigcup_{i<\nu} W(\one_{\mathbb{Q}_i})$.  
Analogously if $x\neq \emptyset$, say $x=(i,p)$,  then $m((i,p),(i,q))=(i,m(p,q))$ and thus, in particular,  $W_n(i,p)=\{i\}\times W_n(p)$. Since $\nu<\mu$, this yields clauses (4) and (5).
\item[(6)] This is obvious.
\item[(7)] Let $U\s P$ be a $0$-open set and fix $x\in P$ and $n<\omega$. If $x\neq\emptyset$,
denote $(i,p):=x$. Otherwise, let $(i,p):=(0,\one_{\mathbb P_0})$. In both cases, $(i,p)\le^0 x$.
Now, it is not hard to check that $U_i:=\{q\in Q_i\mid (i,q)\in U\}$ is also $0$-open. Since $(\mathbb{Q}_i,\lh_i, c_i)$ is $\Sigma$-Prikry we may find $q\in (Q_i)^p_0$ such that either $(Q_i)^q_n\s U_i$ or $(Q_i)^q_n\cap U_i=\emptyset$. Set $y:=(i,q)$. Clearly $y\le^0 x$. If $P^q_n\cap U\neq \emptyset$ then clearly $(Q_i)_n^q\cap U_i\neq \emptyset$, hence $(Q_i)_n^q\s U_i$, and thus $P^q_n\s U$.\qedhere
\end{enumerate}
\end{proof}

\section{Forking projections}\label{sectionforking}

In this section, we introduce the notion of \emph{forking projection} which will play a key role in Section~\ref{killingone}.

\begin{definition}\label{forking} Suppose that $(\mathbb P,\lh_{\mathbb P},c_{\mathbb P})$ is a $\Sigma$-Prikry triple,
$\mathbb A=(A,\unlhd)$ is a notion of forcing,
and $\lh_{\mathbb A}$ and $c_{\mathbb A}$ are functions with $\dom(\lh_{\mathbb A})=\dom(c_{\mathbb A})=A$.

A pair of functions $(\pitchfork,\pi)$ is said to be a \emph{forking projection} from $(\mathbb A,\lh_{\mathbb A})$ to $(\mathbb P,\lh_{\mathbb P})$ 
iff all of the following hold:
\begin{enumerate}
\item\label{frk1}  $\pi$ is a projection from $\mathbb A$ onto $\mathbb P$, and $\lh_{\mathbb A}=\lh_{\mathbb P}\circ\pi$;
\item\label{frk0}  for all $a\in A$, $\fork{}{a}$ is an order-preserving function from $(\cone{\pi(a)},\le)$ to $(\conea{a},\unlhd)$;
\item\label{frk3} for all $p\in P$, $\{ a \in A\mid \pi(a)=p\}$ admits a greatest element, which we denote by $\myceil{p}{\mathbb A}$;
\item\label{frk4}  for all $n,m<\omega$ and $b\unlhd^{n+m} a$, $m(a,b)$ exists and satisfies:
 $$m(a,b)=\fork{}{a}(m(\pi(a),\pi(b)));$$
\item\label{frk5} for all $a\in A$ and $r\le\pi(a)$,   $\pi(\fork{}{a}(r))=r$;
\item\label{frk6} for all $a\in A$ and $r\le\pi(a)$, $a=\myceil{\pi(a)}{\mathbb A}$ iff $\fork{}{a}(r)=\myceil{r}{\mathbb A}$;
\item\label{frk7} for all $a\in A$, $a'\unlhd^0 a$ and $r\le^0 \pi(a')$, $\fork{}{a'}(r)\unlhd\fork{}{a}(r)$.
\setcounter{condition}{\value{enumi}}
\end{enumerate}

The pair $(\pitchfork,\pi)$ is said to be a forking projection from 
$(\mathbb A,\lh_{\mathbb A},c_{\mathbb A})$ to $(\mathbb P,\lh_{\mathbb P},c_{\mathbb P})$ 
iff, in addition to all of the above, the following holds:
\begin{enumerate}
\setcounter{enumi}{\value{condition}}
\item\label{frk2}  for all $a,a'\in A$, if $c_{\mathbb A}(a)=c_{\mathbb A}(a')$, then $c_{\mathbb P}(\pi(a))=c_{\mathbb P}(\pi(a'))$ and, for all $r\in P_0^{\pi(a)}\cap P_0^{\pi(a')}$, $\fork{}{a}(r)=\fork{}{a'}(r)$.
\end{enumerate}
\end{definition}
\begin{example} Suppose that $(\mathbb P,\lh_{\mathbb P},c_{\mathbb P})$ is any $\Sigma$-Prikry triple and that $\mathbb Q$ is any notion of forcing with a greatest element $\one_{\mathbb Q}$.
Let $\mathbb A=(A,\unlhd)$ be the product forcing $\mathbb P\times \mathbb Q$. Define $\pi:A\rightarrow P$ via $\pi(p,q):=p$,
and, for each $a=(p,q)$ in $A$, define $\fork{}{a}:\cone{p}\rightarrow\conea{a}$ via $\fork{}{a}(r):=(r,q)$.
Set $\lh_{\mathbb A}:=\lh_{\mathbb P}\circ\pi$.
Define $c_{\mathbb A}:A\rightarrow\rng(c_{\mathbb P})\times Q$ via $c_{\mathbb A}(p,q):=(c_{\mathbb P}(p),q)$.
Then $\myceil{p}{\mathbb A}=(p,\one_{\mathbb Q})$, $w((p,q),(p',q'))=(w(p,p'),q)$,
and the pair $(\pitchfork,\pi)$ is a forking projection from $(\mathbb A,\lh_{\mathbb A},c_{\mathbb A})$ to $(\mathbb P,\lh_{\mathbb P},c_{\mathbb P})$.
\end{example}

\begin{lemma}\label{forkingfacts} Suppose that $(\pitchfork,\pi)$ is a forking projection from $(\mathbb A,\lh_{\mathbb A})$ to $(\mathbb P,\lh_{\mathbb P})$.
Let $a\in A$.
\begin{enumerate}
\item  $\fork{}{a}\restriction W(\pi(a))$ forms a bijection from $W(\pi(a))$ to $W(a)$;
\item for all $n<\omega$ and $r\le^n\pi(a)$,  $\fork{}{a}(r)\in A_n^a$.
\end{enumerate}
\end{lemma}
\begin{proof} (1) By Clauses \eqref{frk4} and \eqref{frk5} of Definition~\ref{forking}.

(2) By Clauses \eqref{frk1}, \eqref{frk0} and \eqref{frk5} of Definition~\ref{forking}.
\end{proof}

\begin{lemma}\label{zer0pen} Suppose that $(\pitchfork,\pi)$ is a forking projection from $(\mathbb A,\lh_{\mathbb A})$ to $(\mathbb P,\lh_{\mathbb P})$.
Let $U\subseteq A$ and $a\in A$. Denote $U_{a}:=U\cap(\conea{a})$.
\begin{enumerate}
\item If $U_a$ is $0$-open, then so is $\pi[U_{a}]$;
\item If $U_{a}$ is dense below $a$, then $\pi[U_{a}]$ is dense below $\pi(a)$.
\end{enumerate}
\end{lemma}
\begin{proof}
(1) Suppose $U_a$ is $0$-open. To see that $\pi[U_{a}]$ is $0$-open, let $p\in \pi[U_{a}]$ and $p'\le^0 p$ be arbitrary.
Find $b\in U_{a}$ such that $\pi(b)=p$ and set $b':=\fork{}{b}(p')$.
Clearly, $b'$ is well-defined and by Definition~\ref{forking}\eqref{frk5},  $b'\unlhd^0 b$, so that, by $0$-openness of $U_{a}$, $b'\in U_{a}$.
Again, Definition~\ref{forking}\eqref{frk5} yields $\pi(b')=\pi(\fork{}{b}(p'))=p'$, thus
$p'\in\pi[U_{a}]$, as desired.

(2) Suppose that $U_{a}$ is dense below $a$.
To see that $\pi[U_{a}]$ is dense below $\pi(a)$, let $p\le \pi(a)$ be arbitrary.
Since, by Definition~\ref{forking}\eqref{frk1}, $\pi$ is a projection from $\mathbb{A}$ to $\mathbb{P}$,
we may find $a^*\unlhd a$ such that $\pi(a^*)\le p$.
As $U_{a}$ is dense below $a$, we may then find $a^\star\unlhd a^*$ in $U_{a}$.
Clearly, $\pi(a^\star)\le p$.
\end{proof}

Throughout the rest of this section, suppose that:
\begin{itemize}
\item $\mathbb P=(P,\le)$ is a notion of forcing with a greatest element $\one_{\mathbb P}$;
\item $\mathbb A=(A,\unlhd)$ is a notion of forcing with a greatest element $\one_{\mathbb A}$;
\item $\Sigma=\langle \kappa_n\mid n<\omega\rangle$ is a non-decreasing sequence of regular uncountable cardinals,
converging to some cardinal $\kappa$, and  $\mu$ is a cardinal such that $\one_{\mathbb P}\Vdash_{\mathbb P}\check\mu=\check\kappa^+$;
\item $\lh_{\mathbb P}$ and $c_{\mathbb P}$ are functions witnessing that $(\mathbb P,\lh_{\mathbb P},c_{\mathbb P})$ is a $\Sigma$-Prikry;
\item $\lh_{\mathbb A}$ and $c_{\mathbb A}$ are functions with $\dom(\lh_{\mathbb A})=\dom(c_{\mathbb A})=A$;
\item $(\pitchfork,\pi)$ is a forking projection from $(\mathbb A,\lh_{\mathbb A},c_{\mathbb A})$ to $(\mathbb P,\lh_{\mathbb P},c_{\mathbb P})$.
\end{itemize}

We shall now go over each of the clauses of Definition~\ref{SigmaPrikry} and collect sufficient conditions for the triple $(\mathbb A,\lh_{\mathbb A},c_{\mathbb A})$ to be $\Sigma$-Prikry, as well.

\begin{lemma}\label{C2ASigmaPrikry} $(\mathbb A,\lh_{\mathbb A})$ is a graded poset.
\end{lemma}
\begin{proof} For all $a,b\in A$, $b\unlhd a\implies\pi(b)\le\pi(a)\implies\lh_\mathbb{A}(b)=\lh_{\mathbb P}(\pi(b))\geq\lh_{\mathbb P}(\pi(a))=\lh_\mathbb{A}(a)$.
In addition, as $(\mathbb P,\lh_{\mathbb P})$ is a graded poset, for any given $a\in A$, we may pick $r\in P_1^{\pi(a)}$.
By Lemma~\ref{forkingfacts}(2), then, $\fork{}{a}(r)$ witnesses that $A^a_1$ is non-empty.
\end{proof}

\begin{lemma}\label{313} Let $n<\omega$.
Suppose that for every directed family $D$ of conditions in $\mathbb A_n$ with $|D|<\kappa_n$,
if the map $d\mapsto\pi(d)$ is constant over $D$, then $D$ admits a lower bound in $\mathbb A_n$.

Then $\mathbb{A}_n$ is $\kappa_n$-directed-closed.
\end{lemma}
\begin{proof} Suppose that $E$ is a given directed family in $\mathbb A_n$ of size less than $\kappa_n$.
In particular, $\{\pi(e)\mid e\in E\}$ is a directed family in $\mathbb P_n$ of size less than $\kappa_n$; hence, by Definition~\ref{SigmaPrikry}\eqref{c2},
we may find a lower bound for it (in $\mathbb P_n$), say, $r$.
Put $D:=\{ \fork{}{e}(r)\mid e\in E\}$.
By Lemma~\ref{forkingfacts}(2), $D$ is a family of conditions in $\mathbb A_n$ with $|D|<\kappa_n$.
By Definition~\ref{forking}\eqref{frk5}, the map $d\mapsto\pi(d)$ is constant (indeed, with value $r$) over $D$.
\begin{claim} $D$ is directed.
\end{claim}
\begin{proof} Given $d_0,d_1\in D$, fix $e_0,e_1\in E$ such that $d_i=\fork{}{e_i}(r)$ for all $i<2$.
As $E$ is directed, let us pick $e^*\in E$ such that $e^*\unlhd e_0,e_1$. Put $d^*:=\fork{}{e^*}(r)$, so that $d^*\in D$.
Then, by Definition~\ref{forking}\eqref{frk7}, $d^*\unlhd d_0,d_1$.
\end{proof}
Now, by the hypothesis of the lemma, we may pick  a lower bound for $D$ (in $\mathbb A_n$), say, $b$.
By Definition~\ref{forking}\eqref{frk0}, for all $a\in E$, $b\unlhd\fork{}{a}(r)\unlhd a$,
and hence $b$ is a a lower bound for $E$. 
\end{proof}

\begin{lemma}\label{C3ASigmaPriky} For all $a,a'\in A$, if $c_{\mathbb A}(a)=c_{\mathbb A}(a')$, then $A_0^a\cap A_0^{a'}$ is non-empty.

In particular, if $|\rng(c_{\mathbb A})|\le\mu$, then $\mathbb A$ is $\mu^+$-2-linked.
\end{lemma}
\begin{proof}
By Definition~\ref{forking}\eqref{frk2}, $c(\pi(a))=c(\pi(a'))$.
Since $(\mathbb{P},\lh_{\mathbb P},c_{\mathbb P})$ is $\Sigma$-Prikry, Definition~\ref{SigmaPrikry}\eqref{c1} guarantees the existence of some $r\in P_0^{\pi(a)}\cap P_0^{\pi(a')}$ and thus,
again by Definition~\ref{forking}\eqref{frk2}, $\fork{}{a}(r)=\fork{}{a'}(r)$.
Finally, Lemma~\ref{forkingfacts}(2) yields that this common value is in $A_0^a\cap A_0^{a'}$, as desired.
\end{proof}

\begin{lemma}\label{lemma48} For all $a\in A$, $n,m<\omega$ and $b\unlhd^{n+m}a$, $m(a,b)$ exists.
\end{lemma}
\begin{proof} This is covered by Definition~\ref{forking}\eqref{frk4}.
\end{proof} 

\begin{lemma}\label{lemma49}
For all $a\in A$,  $|W(a)|<\mu$.
\end{lemma}
\begin{proof}  This follows from Lemma~\ref{forkingfacts}(1) and Definition~\ref{SigmaPrikry}\eqref{csize} for $(\mathbb{P},\ell_{\mathbb P},c_{\mathbb P})$.
\end{proof}

\begin{lemma}\label{lemma410}
For all $a'\le a$ in $A$, $b\mapsto w(a,b)$ forms an order-preserving map from $W(a')$ to $W(a)$.
\end{lemma}
\begin{proof} Fix an arbitrary pair $b'\unlhd b$ in $W(a')$, and let us show that $w(a,b')\unlhd w(a,b)$.
By Definition~\ref{forking}\eqref{frk4} with $m=0$, $w(a,b')=\fork{}{a}(w(\pi(a),\pi(b')))$ and $w(a,b)=\fork{}{a}(w(\pi(a),\pi(b)))$.
On the other hand, $\pi$ is a projection, in particular order-preserving, hence $\pi(b')\le \pi(b)$, and also both such conditions extend $\pi(a)$.
By Definition~\ref{SigmaPrikry}\eqref{itsaprojection} for $(\mathbb{P},\ell_{\mathbb P},c_{\mathbb P})$, $w(\pi(a), \pi(b'))\le w(\pi(a),\pi(b))$,
and thus, appealing to Definition~\ref{forking}\eqref{frk7}, it follows that
$$\fork{}{a}(w(\pi(a),\pi(b')))\unlhd \fork{}{a}(w(\pi(a),\pi(b))),$$ 
which yields the desired result.
\end{proof}

\begin{definition} The forking projection $(\pitchfork,\pi)$ is said to have the \emph{mixing property} 
iff for all $a\in A$, $n<\omega$, $q\le^0\pi(a)$, 
and a function $g:W_n(q)\rightarrow\conea{a}$ such that $\pi\circ g$ is the identity map,\footnote{Equivalently,
a function $g:W_n(q)\rightarrow A$ such that $g(r)\unlhd a$ and $\pi(g(r))=r$ for every $r\in W_n(q)$.}
there exists $b\unlhd^0 a$ with $\pi(b)=q$ such that $\fork{}{b}(r)\unlhd^0 g(r)$ for every $r\in W_n(q)$.
\end{definition}

\begin{lemma}\label{corollary320} Suppose that $(\pitchfork,\pi)$ has the mixing property.
Let $U\subseteq A$ be a $0$-open set. Then, for all $a\in A$ and $n<\omega$, there is $b\unlhd^0 a$ such that, either $A^{b}_n\cap U=\emptyset$ or $A^{b}_n\subseteq U$.
\end{lemma}
\begin{proof}
Let $a\in A$ and $n<\omega$. Set $U_a:=U\cap (\conea{a})$, $\bar U:=\pi[U_a]$, and $p:=\pi(a)$.
By Lemma~\ref{zer0pen}(1), $\bar U$ is $0$-open.
Since $(\mathbb{P},\lh_{\mathbb P},c_{\mathbb P})$ is $\Sigma$-Prikry, we now appeal to Definition~\ref{SigmaPrikry}\eqref{c6} and find $q\le^0 p$ such that,
either $P^q_n\cap  \bar U=\emptyset$ or $P^q_n\subseteq \bar U$.

\begin{claim} If $P^{q}_n\cap\bar U=\emptyset$, then there exists $b\unlhd^0 a$ with $\pi(b)={q}$ such that $A^b_n\cap U=\emptyset$.
\end{claim}
\begin{proof}
Suppose that $P^{q}_n\cap\bar U=\emptyset$. Set $b:=\fork{}{a}({q})$, so that $b\unlhd a$ and $\pi(b)={q}$.
As $\lh_{\mathbb A}(b)=\lh_{\mathbb P}(q)=\lh_{\mathbb A}(a)$, we moreover have $b\unlhd^0 a$.
Finally, since  $d\in A^b_n\cap U\implies\pi(d)\in P^{q}_n\cap\bar U$, we infer that $A^b_n\cap U=\emptyset$.
\end{proof}
\begin{claim} If $P^{q}_n\s\bar U$, then there exists  $b\unlhd^0 a$ with $\pi(b)=q$ such that $A^b_n\s U$.
\end{claim}
\begin{proof} Suppose that $P^{q}_n\s\bar U$.
So, for every $r\in P^{q}_n$, we may pick $a_r\in U_a$ such that $\pi(a_r)=r$.
Define a function $g: W_n(q)\rightarrow U_a$ via $g(r):=a_r$.
By the mixing property, we now obtain a condition $b\unlhd^0 a$ such that $\fork{}{b}(r)\unlhd^0 g(r)$ for every $r\in W_n(q)$.
As $U$ is $0$-open, it follows that $\fork{}{b}`` W_n(q)\subseteq U$.
By Lemma~\ref{forkingfacts}(1), $W_n(b)=\fork{}{b}`` W_n(q)\s U$; hence, again by $0$-openess of $U$, $A^b_n\subseteq U$, as desired.
\end{proof}
This completes the proof.
\end{proof}

\begin{cor}\label{cor523} 
Suppose that Clauses \eqref{c2} and \eqref{c6} of Definition~\ref{SigmaPrikry} are valid for $(\mathbb A,\lh_{\mathbb A})$. 
If $\one_{\mathbb P}\Vdash_{\mathbb P}``\check\kappa\text{ is singular}"$, then $\one_{\mathbb A}\Vdash_{\mathbb{A}}\check{\mu}=\check\kappa^+$.
\end{cor}
\begin{proof} Suppose that $\one_{\mathbb A}\not\Vdash_{\mathbb{A}}\check{\mu}=\check\kappa^+$.
As $\one_{\mathbb P}\Vdash_{\mathbb{P}}\check{\mu}=\check\kappa^+$ and $\mathbb A$ projects to $\mathbb P$,
this means that there exists $a\in A$ such that $a\Vdash_{\mathbb{A}}|\check\mu|\le|\check\kappa|$.
Towards a contradiction, suppose that $\one_{\mathbb P}\Vdash_{\mathbb{P}}``\check\kappa\text{ is singular}"$. As $\mathbb A$ projects to $\mathbb P$,
it altogether follows that $a\Vdash_{\mathbb{A}}\cf(\check\mu)<\check\kappa$. By Lemma~\ref{l14}\eqref{C2l14}, then, there exists $a'\unlhd a$ with $|W(a')|\geq\mu$,
contradicting Lemma~\ref{lemma48}(2).
\end{proof}

\section{Simultaneous stationary reflection}\label{analysis}

\begin{definition}
For cardinals $\theta<\mu=\cf(\mu)$, and stationary subsets $S,T$ of $\mu$,
the principle $\refl({<}\theta,S,T)$ asserts that for every collection $\mathcal S$ of stationary subsets of $S$,
with $|\mathcal S|<\theta$ and $\sup(\{\cf(\alpha)\mid \alpha\in\bigcup\mathcal S\})<\sup(S)$,
the set $T\cap\bigcap_{S\in\mathcal S}\cap\tr(S)$ is non-empty.

We write $\refl({<}\theta,S)$ for $\refl({<}\theta,S,\mu)$ and $\refl(\theta, S)$ for $\refl({<}\theta^+, S)$.\footnote{Where, for $\theta$ finite, $\theta^+$ stands for $\theta+1$.}
\end{definition}
\begin{definition}[Shelah, {\cite[Definition~5.1, p.~85]{ShelahBook}}]
For infinite cardinals $\mu\geq\nu\geq\theta$, define
$$\cov(\mu, \nu, \theta, 2):=\min\{|\mathcal{A}|\mid \mathcal{A}\s [\mu]^{<\nu}\,\forall X\in [\mu]^{<\theta}\,\exists A\in \mathcal{A}(X\s A)\}.$$
\end{definition}

The following proposition is implicit in the work of Solovay on the Singular Cardinal Hypothesis (SCH).

\begin{prop}
Suppose $\refl({<}\theta,S,E^\mu_{<\nu})$ holds for a stationary $S\s \mu$ and some cardinal $\nu\in \mu$.
Then $\cov(\mu,\nu,\theta,2)=\mu$.
\end{prop}
\begin{proof} Let $\langle S_i\mid i<\mu\rangle$ be a partition of $S$ into mutually disjoint stationary sets.
Put $T:=\{\alpha<\mu\mid \omega<\cf(\alpha)<\nu\}$.
Set $\mathcal A:=\{ A_\alpha\mid \alpha\in T\}$,
where for each $\alpha\in T$, $A_\alpha:=\{ i<\mu\mid S_i\cap\alpha\text{ is stationary}\}$.
Since each $\alpha\in T$ admits a club $C_\alpha$ of order-type ${<}\nu$, and $C_\alpha\cap S_i\neq\emptyset$ for all $i\in A_\alpha$,
while $S_i\cap S_j=\emptyset$ for all $i<j<\mu$,
we get that $\mathcal A\s [\mu]^{<\nu}$.

By $\refl({<}\theta,S,E^\mu_{<\nu})$, for every $X\in[\mu]^{<\theta}$,
there must exist some $A\in\mathcal A$ such that $X\s A$. Altogether, $\mathcal A$ witnesses that $\cov(\mu,\nu,\theta,2)=\mu$.
\end{proof}

Note that for every singular strong limit $\kappa$,  $\cov(\kappa^+,\kappa,(\cf(\kappa))^+,2)=2^\kappa$.
In particular:

\begin{cor}\label{prop42}  If $\kappa$ is a singular strong limit cardinal admitting a stationary subset $S\s\kappa^+$ for which $\refl(\cf(\kappa),S)$ holds,
then  $2^\kappa=\kappa^+$.\qed
\end{cor}
Throughout the rest of this section, suppose that $(\mathbb P,\lh,c)$ is a given $\Sigma$-Prikry notion of forcing.
Denote $\mathbb{P}=(P,\le)$ and $\Sigma=\langle \kappa_n\mid n<\omega\rangle$. Also, define $\kappa$ and $\mu$ as in Definition~\ref{SigmaPrikry}.
Our universe of sets is denoted by $V$, and we
write $\Gamma:=\{\alpha<\mu\mid \omega<\cf^V(\alpha)<\kappa\}$.\footnote{All findings of the analysis in this section goes through if we replace $\mu$ by a regular cardinal $\nu\geq\mu$
and replace $\Gamma$ by $\{\alpha<\nu\mid \omega<\cf^V(\alpha)<\kappa\}$.}

\begin{lemma}\label{keylemma}
Suppose that $r^*\in P$ and that $\tau$ is a $\mathbb P$-name.
For all $n<\omega$, write $\dot{T}_n:=\{(\check\alpha,p)\mid (\alpha,p)\in\mu\times P_n\ \&\ p\Vdash_{\mathbb P}\check\alpha\in\tau\}$.
Then one of the following holds:
\begin{enumerate}
\item $D:=\{p\in P\mid (\forall q\le p)~q\Vdash_{\mathbb{P}_{\lh(q)}}``\dot{T}_{\lh(q)}\text{ is stationary}"\}$ is open and dense below $r^*$;\footnote{Recall that we identify each of the $\mathbb P_n$'s with its separative quotient.}
\item There exist $r^\star\le r^*$ and $I\in[\omega]^\omega$ such that,
for all $q\le r^\star$ with $\lh(q)\in I$,
$$q\Vdash_{\mathbb{P}_{\lh(q)}}``\dot{T}_{\lh(q)}\text{ is nonstationary}".$$
\end{enumerate}
\end{lemma}
\begin{proof}
$D$ is clearly open. Suppose that $D$ is not dense below $r^*$.
Then, we may pick some condition $p^*\le r^*$ such that, for all $p\le p^*$, there is $q\le p$, such that $q\not\Vdash_{\mathbb{P}_{\lh(q)}}``\dot{T}_{\lh(q)}\text{ is stationary}"$, i.e.,
there exists $q'\le q$ in $\mathbb P_{\lh(q)}$ such that $q'\Vdash_{\mathbb{P}_{\lh(q)}}``T_{\lh(q)}\text{ is nonstationary}"$.
Hence, for all $p\le p^*$, there is $q'\le  p$, such that $q'\Vdash_{\mathbb{P}_{\lh(q)}}``\dot{T}_{\lh(q)}\text{ is nonstationary}"$.
In other words, the $0$-open set $E:=\{q\in\mathbb P\mid q\Vdash_{\mathbb{P}_{\lh(q)}}``\dot{T}_{\lh(q)}\text{ is nonstationary}"\}$ is dense below $p^*$.

Now, define a $0$-open coloring  $d: P\rightarrow 2$ via $d(q):=1$ iff $q\in E$.
By virtue of Lemma~\ref{RamseyPrikry}, find $r^\star\le^0 p^*$ such that $\cone{r^\star}$  is a set of indiscernibles for $d$.
Note that as $E$ is dense below $r^\star$, Clause~(\ref{c4}) of Definition~\ref{SigmaPrikry}
entails that the set $I:=\{\lh(q')\mid q'\le r^\star\ \&\ q'\in E\}$ must be infinite.
Finally, as $\cone{r^\star}$  is a set of indiscernibles for $d$,
for all $q\le r^\star$ with $\lh(q)\in I$, we indeed have $q\in E$.
\end{proof}

\begin{lemma}\label{lemma9} Suppose that $r^\star\in P$, $I\in[\omega]^\omega$,
and $\langle \dot{C}_n\mid n\in I\rangle$ is a sequence such that,
for all $q\le r^\star$ with $\lh(q)\in I$, we have:
$$q\Vdash_{\mathbb{P}_{\lh(q)}}``\dot{C}_{\lh(q)}\text{ is a club in }\check\mu".$$

Consider the $\mathbb P$-name $\dot Y:=\{(\check \alpha,q)\mid (\alpha,q)\in R\}$, where\label{definitionOFr}
$$R:=\{(\alpha,q)\in\mu\times P\mid q\le r^\star\ \&\ \forall r\le q[\lh(r)\in I\rightarrow r\Vdash_{\mathbb{P}_{\lh(r)}}\check\alpha\in \dot{C}_{\lh(r)}]\}.$$

Suppose $G$ is $\mathbb P$-generic over $V$, with $r^\star\in G$. Let $Y$ be the interpretation
of $\dot Y$ in $V[G]$. Then:
\begin{enumerate}
\item $V[G]\models Y\text{ is unbounded in }\mu$;
\item $V[G]\models \acc^+(Y)\cap\Gamma\s Y$.
\end{enumerate}
\end{lemma}
\begin{proof} We commence with a claim.
\begin{claim}\label{claim72} For every $p\le r^\star$ and $\gamma<\mu$, there exist $\bar p\le^0 p$ and $\bar\gamma\in(\gamma,\mu)$ such that,
for every $q\le\bar p$ with $\lh(q)\in I$,  $q\Vdash_{\mathbb P_{\lh(q)}}``\dot{C}_{\lh(q)}\cap(\gamma,\bar\gamma)\text{ is non-empty}"$.
\end{claim}
\begin{proof}
Given $p$ and $\gamma$ as above, write:
$$D_{p,\gamma}:=\{q\in\mathbb P\mid q\le p\ \&\ \lh(q)\in I\ \&\ \exists \gamma'>\gamma(q\Vdash_{\mathbb P_{\lh(q)}}\check\gamma'\in\dot{C}_{\lh(q)})\}.$$
Note that $I_{p,\gamma}:=\{\lh(q)\mid q\in D_{p,\gamma}\}$ is equal to $I\setminus\lh(p)$.\footnote{By standard facts about forcing, if $\mathbb Q$ is a notion of forcing, and $q$ is a condition in $\mathbb Q$
forcing that $\dot C$ is some cofinal subset of a cardinal $\mu$,
then for every ordinal $\gamma<\mu$, there exists an extension $q'$ of $q$ and some ordinal $\gamma'$ above $\gamma$
such that $q'\Vdash_{\mathbb Q}\check\gamma'\in\dot C$.}  Let $d: P\rightarrow 2$ be defined via $d(r):=1$ iff $r\in {D_{p,\gamma}}$. As ${D_{p,\gamma}}$ is $0$-open
we get from Lemma~\ref{RamseyPrikry} a condition $\bar p\le^0 p$ such that $\cone{\bar p}$ is a set of indiscernibles for $d$. Thereby, for all $n<\omega$,
if $P^{\bar p}_n\cap {D_{p,\gamma}}\neq\emptyset$, then $P^{\bar p}_n\s { D_{p,\gamma}}$.
 As $\bar p\le p$, $I_{p,\gamma}=I\setminus\lh(p)$,
and $W_n(\bar p)\s P^{\bar p}_n$ for all $n<\omega$,
we get in particular that $A_n:=W_{n-\lh(\bar p)}(\bar p)$ is a subset of $D_{p,\gamma}$ for all $n\in I\setminus\lh(p)$.

For all $n\in I\setminus\lh(p)$ and $r\in A_n$, fix $\gamma_r\in(\gamma,\mu)$ such that $$r\Vdash_{\mathbb P_{\lh(r)}}\gamma_r\in\dot{C}_{\lh(r)}.$$

By Definition~\ref{SigmaPrikry}(\ref{csize}), $|\bigcup_{n\in I\setminus\lh(p)}A_n|<\mu$,
so that $\bar\gamma:=\sup\{\gamma_r\mid r\in\bigcup_{n\in I\setminus\lh(p)}A_n\}+1$ is ${<}\mu$.

Now, let  $q\le \bar p$ with length in $I$ be arbitrary. As $I_{p,\gamma}=I\setminus\lh(p)$, we have $\lh(q)\in I_{p,\gamma}$.
In particular, $P^{\bar p}_{\lh(q)-\lh(\bar p)}\cap{D_{p,\gamma}}\neq\emptyset$,
and thus $A_{\lh(q)}\s D_{p,\gamma}$. Pick $r\in A_{\lh(q)}$ with $q\le r$. Then $r\Vdash_{\mathbb P_{\lh(r)}}\gamma_r\in\dot{C}_{\lh(r)}$.
In particular, $q\Vdash_{\lh(q)}``\dot{C}_{\lh(q)}\cap(\gamma,\bar\gamma)\text{ is non-empty}"$.
\end{proof}
Now, let $G$ be a $\mathbb P$-generic with $r^\star\in G$.
Of course, the interpretation of $\dot Y$ in $V[G]$ is $$Y:=\{\alpha<\mu \mid (\exists q\in G)(\forall r\le q)[\lh(r)\in I\rightarrow r\Vdash_{\mathbb{P}_{\lh(r)}}\check\alpha\in \dot{C}_{\lh(r)}]\}.$$
\begin{claim}\label{claim73}
\begin{enumerate}
\item $Y$ is unbounded in $V[G]$;
\item $\acc^+(Y)\cap\Gamma\s Y$.
\end{enumerate}
\end{claim}
\begin{proof} (1) We run a density argument in $V$. Let $p\le r^\star$ and $\gamma<\mu$ be arbitrary.
By an iterative application of Claim \ref{claim72}, we find a $\le_0$-decreasing sequence of conditions in $\mathbb P$, $\langle p_n\mid n<\omega\rangle$,
and an increasing sequence of ordinals below $\mu$, $\langle\gamma_n\mid n<\omega\rangle$,
such that $p_0\le^0 p$, $\gamma_0=\gamma$, and such that for every $n<\omega$ and every $q\le p_n$ with $\lh(q)\in I$, we have that $q\Vdash_{\mathbb P_{\lh(q)}}``\dot{C}_{\lh(q)}\cap(\gamma_{n},\gamma_{n+1})\text{ is non-empty}"$.

By Definition~\ref{SigmaPrikry}\eqref{c2}, $\mathbb P_{\lh(p)}$ is $\sigma$-closed,
so let $q^*$ be a lower bound for $\langle q_n\mid n<\omega\rangle$.
Put $\gamma^*:=\sup_{n<\omega}\gamma_n$. Then for every $r\le q^*$ with length in $I$,
 we have $r\Vdash_{\mathbb P_{\lh(r)}}\gamma^*\in\dot{C}_{\lh(r)}$. That is, $q^*$ witnesses that $\gamma^*\in Y\setminus\gamma$.

(2) Suppose that $\alpha\in\acc^+(Y)\cap\Gamma$.
Set $\eta:=\cf^V(\alpha)$, and pick a large enough ${k}<\omega$ such that $\eta<\kappa_{{k}}$.
Fix $p\in G$ such that $p\le r^\star$, $p\Vdash\check\alpha\in\acc^+(\dot{Y})$, and $\lh(p)\geq k$.

Work in $V$. Let $\langle\alpha_j\mid j<\eta\rangle$ be an increasing cofinal sequence in $\alpha$.
For each $j<\eta$, consider the set $D_j:=\{q\in P\mid \exists\gamma\in(\alpha_j,\alpha)~q\Vdash_{\mathbb{P}} \check\gamma\in\dot{Y}\}$.
Clearly, $D_j$ is open and dense below $p$. We claim that the intersection $\bigcap_{j<\eta}D_j$ is dense below $p$, as well.
To this end, let $p'\le p$ be arbitrary. For each $j<\eta$, $D_j$ is $0$-open and dense below $p'$,
so since $\eta<\kappa_k\le\kappa_{\lh(p')}$, we obtain from Corollary~\ref{l6}(2) and Definition~\ref{SigmaPrikry}(\ref{c2}),
a $\le_0$-decreasing sequence $\langle q_j\mid j\leq\eta\rangle$ along with a sequence of natural numbers $\langle n_j\mid j<\eta\rangle$
such that $q_0\le^0 p'$ and $P^{q_j}_{n_j}\subseteq D_j$ for all $j<\eta$.
Let $p'':=q_\eta$.
As $\eta=\cf^V(\alpha)>\omega$, we may pick a cofinal $J\s\eta$ for which $\{ n_j\mid j\in J\}$ is a singleton, say, $\{n\}$.
Then $P^{p''}_{n}\s \bigcap_{j\in J}P^{q_j}_{n_j}\s \bigcap_{j\in J}D_j=\bigcap_{j<\eta}D_j$.
Thus, the latter contains an element extending $p''$, which extends $p'$.

Fix $q\in G\cap \bigcap_{j<\eta}D_j$ extending $p$ and let us show that $q$ witnesses that $\alpha$ is in $Y$.
That is, we shall verify that, for all $r\le q$ with $\lh(r)\in I$, $r\Vdash_{\mathbb P_{\lh(r)}}\check\alpha\in\dot C_{\lh(r)}$.
First, notice that for all $j<\eta$, there exists some $\gamma_j\in(\alpha_j,\alpha)$ such that $q\Vdash_{\mathbb{P}} \check\gamma_j\in \dot{Y}$.
Now let $r\le q$  with $\lh(r)\in I$ be arbitrary and notice that $r\Vdash_{\mathbb P_{\lh(r)}}\check\gamma_j\in\dot C_{\lh(r)}$ for all $j<\eta$,
hence  $r\Vdash_{\mathbb P_{\lh(r)}}\check\alpha\in\dot C_{\lh(r)}$.
\end{proof}
This completes the proof of Lemma~\ref{lemma9}.
\end{proof}

\begin{lemma}\label{lemma10} Suppose that $r^*\in P$ forces that $\tau$ is a $\mathbb P$-name for a stationary subset $T$ of $\Gamma$.
For all $n<\omega$, write $\dot{T}_n:=\{(\check\alpha,p)\mid (\alpha,p)\in\mu\times P_n\ \&\ p\Vdash_{\mathbb P}\check\alpha\in\tau\}$.
Then  $D:=\{p\in P\mid (\forall q\le p)~q\Vdash_{\mathbb{P}_{\lh(q)}}``\dot{T}_{\lh(q)}\text{ is stationary}"\}$ is open and dense below $r^*$.
\end{lemma}
\begin{proof} Suppose not. Then, by Lemma~\ref{keylemma},
let us pick $r^\star\le r^*$ and $I\in[\omega]^\omega$ such that, for all $q\le r^\star$ with $\lh(q)\in I$,
$$q\Vdash_{\mathbb{P}_{\lh(q)}}``\dot{T}_{\lh(q)}\text{ is nonstationary}".$$
Now, for each $n\in I$, we appeal to the maximal principle (also known as the \emph{mixing lemma}) to find a $\mathbb{P}_n$-name $\dot{C}_n$ for a club subset of $\mu$, such that, for all $q\le r^\star$ with $\lh(q)\in I$,
we have $q\Vdash_{\mathbb{P}_{\lh(q)}}\dot{C}_{\lh(q)}\cap\dot{T}_{\lh(q)}=\emptyset$.
Consider the $\mathbb P$-name:
$$\dot Y:=\{(\check \alpha,q)\in\mu\times P \mid q\le r^\star\ \&\ \forall r\le q[\lh(r)\in I\rightarrow r\Vdash_{\mathbb{P}_{\lh(r)}}\check\alpha\in \dot{C}_{\lh(r)}]\}.$$
Let $G$ be $\mathbb P$-generic over $V$, with $r^\star\in G$, and $Y$ be the interpretation of $\dot Y$ in $V[G]$. By Lemma~\ref{lemma9}:
 \begin{enumerate}
\item $V[G]\models Y\text{ is unbounded in }\mu$;
\item $V[G]\models \acc^+(Y)\cap\Gamma\s Y$.
\end{enumerate}
As $r^\star\le r^*$, our hypothesis entails:
 \begin{enumerate}
 \item[(3)] $V[G]\models T\text{ is a stationary subset of }\Gamma$.
\end{enumerate}

So $V[G]\models Y\cap T\neq\emptyset$. Pick $\alpha<\mu$ and $r\in G$
such that $r\Vdash_{\mathbb P}\check\alpha\in\dot Y\cap\tau$. Of course, we may find such $r$ that in addition satisfies $r \le r^\star$ and $\lh(r)\in I$.
By definition of $\dot T_{\lh(r)}$,  the ordered-pair $(\check\alpha,r)$ is an element of the name $\dot T_{\lh(r)}$. In particular,
$r\Vdash_{\mathbb P_{\lh(r)}}\check\alpha\in\dot T_{\lh(r)}$.

From $r\le r^\star$, $\lh(r)\in I$, and $r\Vdash_{\mathbb P}\check\alpha\in\dot Y$, we have $r\Vdash_{\mathbb P_{\lh(r)}}\check\alpha\in\dot C_{\lh(r)}$.

Altogether $r\Vdash_{\mathbb P_{\lh(r)}}\dot C_{\lh(r)}\cap \dot T_{\lh(r)}\neq\emptyset$, contradicting the choice of $\dot C_{\lh(r)}$.
\end{proof}

Recall that a supercompact cardinal $\chi$ is said to be \emph{Laver-indestructible} iff for every $\chi$-directed-closed notion of forcing $\mathbb{Q}$, $\one_{\mathbb Q}\Vdash_\mathbb{Q}``\chi\text{ is supercompact}"$.
Also recall that for every supercompact cardinal $\chi$ and every regular cardinal $\nu\geq\chi$, $\refl({<}\chi,E^\nu_{<\chi},E^\nu_{<\chi})$ holds.
We refer the reader to \cite{MR2768691} for further details.
For our purpose, we would just need the following:

\begin{lemma}\label{l25} For all $n<\omega$,
if $\kappa_n$ is a Laver-indestructible supercompact cardinal, then $V^{\mathbb P_n}\models \refl({<}\omega,E^\mu_{<\kappa_n},E^\mu_{<\kappa_n})$.\footnote{Note that, as $\mathbb P_n$ is $\kappa_n$-closed,
$(E^\mu_{<\kappa_n})^{V^{\mathbb P_n}}=(E^\mu_{<\kappa_n})^{V}$.}
\end{lemma}
\begin{proof} By Definition~\ref{SigmaPrikry}\eqref{c2}, $\mathbb P_n$ is $\kappa_n$-directed-closed,
and hence $V^{\mathbb P_n}\models``\kappa_n\text{ is supercompact}"$. In particular, $V^{\mathbb P_n}\models\refl({<}\omega,E^\mu_{<\kappa_n},E^\mu_{<\kappa_n})$.
\end{proof}

\begin{lemma}\label{lemma11} Suppose:
\begin{itemize}
\item For all $n<\omega$, $V^{\mathbb P_n}\models\refl({<}\omega,E^\mu_{<\kappa_n},E^\mu_{<\kappa_n})$;
\item $r^*\in P$ forces that $\langle \tau^i\mid i<k\rangle$ is a finite sequence of $\mathbb P$-names for stationary subsets of $(E^\mu_{<\kappa})^V$;
\end{itemize}
Write $\dot{T}_n^i:=\{(\check\alpha,p)\mid (\alpha,p)\in \mu\times P_n\ \&\ p\Vdash_{\mathbb P}\check\alpha\in\tau^i\}$ for all $i<k$ and $n<\omega$.

Suppose $D^i:=\{p\in P\mid (\forall q\le p) q\Vdash_{\mathbb{P}_{\lh(q)}}``\dot{T}^i_{\lh(q)}\text{ is stationary}"\}$ is open and dense below $r^*$ for each $i<k$. Then for every $\mathbb P$-generic $G$ over $V$ with $r^*\in G$,  $\langle T^i\mid i<k\rangle$  reflects simultaneously in $V[G]$.\footnote{ $\langle T^i\mid i<k\rangle$ stands for the $G$-interpretation of the sequence of $\mathbb{P}$-names $\langle \tau^i\mid i<k\rangle$.}
\end{lemma}
\begin{proof} We run a density argument below the condition $r^*$. Given an arbitrary $p_0\le r^*$,
pick $p\in\bigcap_{i<k}D^i$ below $p_0$ and a large enough $m<\omega$ such that
$p\Vdash_\mathbb{P}``\forall i<k(\tau^i\cap E^\mu_{<\kappa_m})\text{ is stationary}"$.
By possibly extending $p$ using Definition~\ref{SigmaPrikry}\eqref{c4}, we may assume that $n:=\lh(p)$ is $\geq m$.
Let $G_n$ be $\mathbb{P}_n$-generic with $p\in G_n$.
As $V[G_n]\models\refl({<}\omega,E^\mu_{<\kappa_n},E^\mu_{<\kappa_n})$,
let us fix some $q\le^0 p$ in $G_n$, and some $\delta\in E^\mu_{<\kappa_n}$ such that $q\Vdash_{\mathbb{P}_n}``\forall i<k(\dot{T}^i_n\cap\delta\text{ is stationary})"$.

In $V$, pick a club $C\s\delta$ of order type $\cf(\delta)$. Note that $|C|<\kappa_n$.
Then for each $i<k$, $q\Vdash_{\mathbb{P}_n}``\dot{T}^i_n\cap C\text{ is stationary in }\delta"$. Working for a moment in $V[G_n]$, write
$A^i:=C\cap (\dot{T}^i_n)_{G_n}$. 
Since $\mathbb{P}_n$ is $\kappa_n$-closed, we may find $r\in P_n$ extending $q$ that,
for all $i<k$, decides $A^i$  to be some ground model stationary subset $B^i$ of $\delta$.
Then, for every $i<k$, $$r\Vdash_{\mathbb{P}_n}``\dot{T}_n^i\cap\delta\text{ contains the stationary set }\check B^i".$$

By definition of the name $\dot T^i_n$, we have that $r\Vdash_{\mathbb{P}}\check B^i\s\tau^i\cap\delta$.
Finally, since $\otp(B^i)\leq\delta<\kappa$, Lemma~\ref{l14}\eqref{C1l14}, $B^i$ remains stationary in $V^{\mathbb P}$ for each $i$.
So, $r\le p_0$, and $r\Vdash_{\mathbb{P}}``\tau^i\cap\delta\text{ is stationary for each }i<k"$.
\end{proof}

\begin{cor}\label{c27}  Suppose $V^{\mathbb P_n}\models\refl({<}\omega,E^\mu_{<\kappa_n},E^\mu_{<\kappa_n})$ for all $n<\omega$.
Then  $V^{\mathbb P}\models \refl({<}\omega,\Gamma)$.
\end{cor}
\begin{proof} Let $r^*$ be a condition in $G$ forcing that $\langle \tau^i\mid i<k\rangle$ is a finite sequence of $\mathbb P$-names for stationary subsets $\langle T^i\mid i<k\rangle$ of $\Gamma$.
For each $i<k$ and each $n<\omega$, write $\dot{T}_n^i:=\{(\check\alpha,p)\mid (\alpha,p)\in(\mu\times P_n)\ \&\ p\Vdash_{\mathbb P}\check\alpha\in\tau^i\}$.
By Lemma~\ref{lemma10}, for each $i<k$, $D^i:=\{p\in P\mid (\forall q\le p) q\Vdash_{\mathbb{P}_{\lh(q)}}``\dot{T}^i_{\lh(q)}\text{ is stationary}"\}$ is open and dense below $r^*$.
Finally, by virtue of Lemma~\ref{lemma11}, $\langle T^i\mid i<k\rangle$  reflects simultaneously in $V[G]$.
\end{proof}

Putting Lemma~\ref{l25} together with Corollary~\ref{c27}, we arrive at the following conclusion.
\begin{cor}\label{c28}
 Suppose that each cardinal in $\Sigma$ is a Laver-indestructible supercompact cardinal. Then $\one\Vdash_\mathbb{P} \refl({<}\omega, \Gamma)$.\qed
\end{cor}

Towards a model $V[G]$ satisfying $\refl({<}\omega,\kappa^+)$, we would need to address the reflection of stationary subsets of $\mu\setminus\Gamma$.
In the special case that $\kappa$ is singular and $\mu=\kappa^+$, the set $\mu\setminus\Gamma$ will be nothing but $(E^{\mu}_\omega)^V$.
It is not hard to verify that in this scenario, $V[G]$ will satisfy $\refl({<}\omega,\kappa^+)$ iff it will satisfy $\refl({<}\omega,\Gamma)+\refl(1,(E^{\mu}_\omega)^V,\Gamma)$.\footnote{The easy proof may be found in \cite{partII}.}
For this, in the next section we shall devise a notion of forcing for killing a given single counterexample to $\refl(1,E^\mu_\omega,\Gamma)$.
Then, in \cite{partII}, we find a mean to iterate it.

\section{Killing one non-reflecting stationary set}\label{killingone}

Throughout this section, suppose that $(\mathbb P,\lh,c)$ is a given $\Sigma$-Prikry notion of forcing.
Denote $\mathbb{P}=(P,\le)$ and $\Sigma=\langle \kappa_n\mid n<\omega\rangle$.
Also, define $\kappa$ and $\mu$ as in Definition~\ref{SigmaPrikry},
and assume that $\one_{\mathbb P}\Vdash_{\mathbb P}``\check\kappa\text{ is singular}"$ and that $\mu^{<\mu}=\mu$.
Our universe of sets is denoted by $V$,
and we assume that, for all $n<\omega$,
$V^{\mathbb P_n}\models \refl(1,E^\mu_{\omega},E^\mu_{<\kappa_n})$.\footnote{In particular, $\kappa_n>\aleph_1$ in $V^{\mathbb P_n}$.}
Write $\Gamma:=\{\alpha<\mu\mid \omega<\cf^V(\alpha)<\kappa\}$.

\begin{lemma}\label{cor9}
Suppose $r^\star\in P$ forces that $\dot T$ is a $\mathbb P$-name for a stationary subset $T$ of $(E^\mu_\omega)^V$ that does not reflect in $\Gamma$.
For each $n<\omega$, write $\dot{T}_n:=\{(\check\alpha,p)\mid (\alpha,p)\in E^\mu_\omega\times P_n\ \&\ p\Vdash_{\mathbb P}\check\alpha\in\dot T\}$.
Then, for every $q\le r^\star$,
we have $q\Vdash_{\mathbb{P}_{\lh(q)}}``\dot{T}_{\lh(q)}\text{ is nonstationary}"$.
\end{lemma}
\begin{proof}
Towards a contradiction, suppose that there exists $q\le r^\star$ such that
$q\not\Vdash_{\mathbb{P}_{\lh(q)}}``\dot{T}_{\lh(q)}\text{ is nonstationary}"$.
Consequently, we may pick $p\le^0 q$ such that $p\Vdash_{\mathbb{P}_{n}}``\dot{T}_{n}\text{ is stationary}"$, for $n:=\lh(q)$.
Let $G_n$ be $\mathbb{P}_n$-generic with $p\in G_n$.
As $V[G_n]\models\refl(1,E^\mu_{\omega},E^\mu_{<\kappa_n})$,
let us fix $p'\le^0 p$ in $G_n$, and some $\delta\in E^\mu_{<\kappa_n}$ of uncountable cofinality such that $p'\Vdash_{\mathbb{P}_n}``\dot{T}_n\cap\delta\text{ is stationary}"$.
As $\mathbb P_n$ is $\kappa_n$-closed, $\delta\in\Gamma$.
In $V$, pick a club $C\s\delta$ of order type $\cf(\delta)$. Note that $|C|<\kappa_n$.
Then, $p'\Vdash_{\mathbb{P}_n}``\dot{T}_n\cap C\text{ is stationary in }\delta"$. Working for a moment in $V[G_n]$, write
$A:=C\cap (\dot{T}_n)_{G_n}$. 
Since $\mathbb{P}_n$ is $\kappa_n$-closed, we may find $r\in P_n$ extending $p'$ that
decides $A$  to be some ground model stationary subset $B$ of $\delta$.
Namely, $$r\Vdash_{\mathbb{P}_n}``\dot{T}_n\cap\delta\text{ contains the stationary set }\check B".$$

By definition of the name $\dot T_n$, we have that $r\Vdash_{\mathbb{P}}\check B\s\dot{T}\cap\delta$.
Finally, as $\otp(B)<\kappa$, we infer from Lemma~\ref{l14}\eqref{C1l14} that $B$ remains stationary in any forcing extension by $\mathbb P$.
So, $r\le p'\le p\le q\le r^\star$, and $r\Vdash_{\mathbb{P}}``\dot{T}\cap\delta\text{ is stationary}"$,
contradicting the fact that $r^\star$ forces $\dot{T}$ to not reflect in $\Gamma$.
\end{proof}

Suppose $r^\star\in P$ forces that $\dot T$ is a $\mathbb P$-name for a stationary subset $T$ of $(E^\mu_\omega)^V$ that does not reflect in $\Gamma$.
We shall devise a $\Sigma$-Prikry notion of forcing $(\mathbb A,\lh_{\mathbb A},c_{\mathbb A})$
such that $\mathbb A=\mathbb A(\mathbb P,\dot{T})$ projects to $\mathbb P$ and kills the stationarity of $T$. 
Moreover, $(\mathbb A,\lh_{\mathbb A},c_{\mathbb A})$ will admit a forking projection to $(\mathbb P,\lh,c)$ with the mixing property.

Here goes.
For all $n<\omega$, write $\dot{T}_n:=\{(\check\alpha,p)\mid (\alpha,p)\in E^\mu_\omega\times P_n\ \&\ p\Vdash_{\mathbb P}\check\alpha\in\dot{T}\}$.
Let $I:=\omega\setminus\lh(r^\star)$.
By Lemma~\ref{cor9}, for all $q\le r^\star$ with $\lh(q)\in I$, $q\Vdash_{\mathbb{P}_{\lh(q)}}``\dot{T}_{\lh(q)}\text{ is nonstationary}"$.
Thus, for each $n\in I$, we may pick a $\mathbb P_n$-name $\dot{C}_n$ for a club subset of $\mu$ such that,
 for all $q\le r^\star$ with $\lh(q)=n$, $$q\Vdash_{\mathbb P_n}\dot T_n\cap\dot C_n=\emptyset.$$
Consider the binary relation $R$ as defined in Lemma~\ref{lemma9} (page~\pageref{definitionOFr}) with respect to $\langle \dot C_n\mid n\in I\rangle$.
A moment reflection makes it clear that, for all $(\alpha,q)\in R$, $q\Vdash_\mathbb{P}\check\alpha\notin\dot{T}$.

\begin{definition}\label{labeled-p-tree} Suppose $p\in P$.
A \emph{labeled $p$-tree} is a function $S:W(p)\rightarrow[\mu]^{<\mu}$ such that for all $q\in W(p)$:
\begin{enumerate}
\item\label{C1ptree} $S(q)$ is a closed bounded subset of $\mu$;
\item\label{C2ptree} $S(q')\supseteq S(q)$ whenever $q'\le q$;
\item\label{C3ptree} $q\Vdash_{\mathbb P} S(q)\cap\dot{T}=\emptyset$;
\item\label{d162}
\label{C4ptree} for all $q'\le q$ in $W(p)$, either $S(q')=\emptyset$ or $(\max(S(q')),q)\in R$.
\end{enumerate}
\end{definition}

\begin{definition}\label{strategy}
For $p\in P$, we say that $\vec S=\langle S_i\mid i\leq\alpha\rangle$ is a \emph{$p$-strategy} iff all of the following hold:
\begin{enumerate}
\item\label{C1pstrategy} $\alpha<\mu$;
\item\label{i3}
\label{C2pstrategy} $S_i$ is a labeled $p$-tree for all $i\leq\alpha$;
\item\label{C3pstrategy} for every $i<\alpha$ and $q\in W(p)$, $S_{i}(q)\sqsubseteq S_{i+1}(q)$;
\item\label{C4pstrategy} for every $i<\alpha$ and a pair $q'\le q$ in $W(p)$,  $(S_{i+1}(q)\setminus S_i(q))\sqsubseteq (S_{i+1}(q')\setminus S_i(q'))$;
\item\label{C5pstrategy} for every  limit $i\leq\alpha$ and $q\in W(p)$, $S_i(q)$ is the ordinal closure of $\bigcup_{j<i}S_j(q)$.
In particular, $S_0(q)=\emptyset$ for all $q\in W(p)$.
\end{enumerate}
\end{definition}

This section centers around the following notion of forcing.
\begin{definition}\label{d20}
Let $\mathbb{A}(\mathbb{P}, \dot{T})$ be the notion of forcing $\mathbb{A}:=(A,\unlhd)$, where:
\begin{enumerate}
\item
\label{C1d20} $(p,\vec S)\in A$ iff $p\in P$, and $\vec S$ is either the empty sequence, or a $p$-strategy;
\item
\label{C2d20} $(p', \vec{S'})\unlhd(p, \vec S)$ iff:
\begin{enumerate}
\item
\label{C2ad20} $p'\le p$;
\item
\label{C2bd20} $\dom(\vec{S'})\geq \dom(\vec S)$;
\item
\label{C2cd20} $S'_i(q)=S_i(w(p,q))$ for all $i\in \dom(\vec S)$ and $q\in W(p')$.
\end{enumerate}
\end{enumerate}

For all $p\in P$, denote $\myceil{p}{\mathbb A}:=(p,\emptyset)$.
\end{definition}
\begin{remark}
The relation $\unlhd$  is well-defined as $w(p,q)\in W(p)$, the domain of the $p$-labeled trees $S_i$.
\end{remark}

It is easy to see that $\one_\mathbb A=\myceil{\one_\mathbb P}{\mathbb A}$.

\begin{lemma}\label{lemma35} For every $\nu\ge\mu$, if $\mathbb P$ is a subset of $H_{\nu}$, then so is $\mathbb A$.
\end{lemma}
\begin{proof} Suppose $\mathbb P\s H_{\nu}$ for a given $\nu\ge\mu$. To prove that $\mathbb A\s H_{\nu}$,
it suffices to show that $A\s H_{\nu}$.
Now, each element of $A$ is a pair $(p,\vec S)$, with $p\in P\s H_{\nu}$ and $\vec S\in{}^{<\mu}({}^{W(p)}[\mu]^{<\mu})$, so, as $\nu\ge\mu$,  it suffices to show that ${}^{W(p)}[\mu]^{<\mu} \s H_{\nu}$.
Any element of ${}^{W(p)}[\mu]^{<\mu}$ is a subset of $W(p)\times[\mu]^{<\mu}$ of size $|W(p)|$ and, in particular, a subset of $H_\nu\times H_\mu$ of size ${<}\mu$ because of Definition~\ref{SigmaPrikry}(\ref{csize}),
so that it is indeed an element of $H_\nu$.
\end{proof}

\begin{lemma}\label{claim244}
Suppose $(p,\vec S)\in A$, where $p$ is compatible with $r^\star$. For every $\epsilon<\mu$,
there exist $\alpha>\epsilon$ and $(q,\vec T)\unlhd(p,\vec S)$ such that, for all $r\in W(q)$, $\dom(\vec T)=\alpha+1$ and $\max(T_\alpha(r))=\alpha$.
\end{lemma}
\begin{proof} Fix $p'\le p,r^\star$.
Define a $p'$-strategy $\vec S'$ with $\dom(\vec{S})=\dom(\vec{S}')$ using Clause~\eqref{C2cd20} of Definition~\ref{d20}, $(p',\vec S')\unlhd (p,\vec S)$.
Next, let $\epsilon<\mu$ be arbitrary. Since $(\mathbb P,\lh,c)$ is $\Sigma$-Prikry, we infer from Definition~\ref{SigmaPrikry}(\ref{csize}) that $|W(p')|<\mu$.
Thus, by possibly extending $\epsilon$, we may assume that $S'_i(q)\s\epsilon$, for all $q\in W(p')$ and $i\in\dom(\vec S')$.

Assume for a moment that $\vec S'\neq \emptyset$ and write $\delta+1:=\dom(\vec S')$.
As $p'\le r^\star$, by the very same proof of Claim~\ref{claim73}(1), we may fix $(\alpha,q)\in R$ with $\alpha>\delta+\epsilon$ and $q\le p'$.
Define $\vec T=\langle T_i:W(q)\rightarrow[\mu]^{<\mu}\mid i\leq\alpha\rangle$ by letting for all $r\in W(q)$ and $i\in\dom(\vec T)$:
$$T_i(r):=\begin{cases}
S'_i(w(p',r)),&\text{if }i\leq\delta;\\
S'_{\delta}(w(p',r))\cup\{\alpha\},&\text{otherwise}.
\end{cases}$$
It is easy to see that $T_i$ is a labeled $q$-tree for each $i\leq \alpha$. By Definitions \ref{strategy} and \ref{d20},
we also have that $(q,\vec T)$ is a condition in $\mathbb A$ and $(q,\vec T)\unlhd(p',\vec S')\unlhd(p,\vec S)$.
Altogether, $\alpha$ and $(q,\vec T)$ are as desired.

In case $\vec{S}=\emptyset$, arguing as before we may find $(\alpha,q)\in R$ with $\alpha>\epsilon$ and $q\le p'$.
Define $\vec{T}=\langle T_i: W(q) \rightarrow[\mu]^{<\mu}\mid i\leq\alpha\rangle$ by letting for all $r\in W(q)$ and $i\in\dom(\vec T)$:
$$T_i(r):=\begin{cases}
\emptyset,&\text{if }i=0;\\
\{\alpha\},&\text{otherwise}.
\end{cases}$$
It is clear that $\vec{T}$ is a $q$-strategy and that $(q,\vec{T})$ is as desired.
\end{proof}

\begin{theorem}\label{purpose}  $(r^\star,\emptyset)\Vdash_{\mathbb A}``\dot{T}\text{ is nonstationary}"$.
\end{theorem}
\begin{proof}
Let $G$ be $\mathbb A$-generic over $V$, with $(r^\star,\emptyset)\in G$. Work in $V[G]$. Let $\bar{G}$ be the induced generic for $\mathbb{P}$ via $\pi$,
so that $r^\star\in\bar G$.

For all $a=(p,\vec S)$ in $G$ and $i\in\dom(\vec S)$, write $d^i_a:=\bigcup\{ S_{i}(q)\mid q\in\bar G\cap W(p)\}$.
Then, let
$$d_a:=\begin{cases}
d_a^{\max(\dom(\vec S))},&\text{if }\vec S\neq\emptyset;\\
\emptyset,&\text{otherwise}.\end{cases}$$

\begin{claim} Suppose that $a=(p,\vec S)$ is an element of $G$.

In $V[\bar G]$, for all $i\in\dom(\vec S)$, the ordinal closure $\cl(d^i_a)$ of $d^i_a$ is disjoint from $T$.
\end{claim}
\begin{proof} Work in $V[\bar G]$.
By Lemma~\ref{lemma7}(1), for all $n<\omega$, there exists a unique element in $\bar G\cap W_n(p)$, which we shall denote by $p_n$.
By Lemma~\ref{lemma7}(2), it follows that $\langle p_n\mid n<\omega\rangle$ is $\le$-decreasing and then, by Definition~\ref{labeled-p-tree},
for each $i\in\dom(\vec S)$, $\langle S_i(p_n)\mid n<\omega\rangle$ is a weakly $\s$-increasing (though, not $\sqsubseteq$-increasing)  sequence of closed sets that converges to $d^i_a$.

We now argue by induction on $i\in\dom(\vec S)$. The base case is trivial, since $d^0_a=\emptyset$.

Now, suppose that the claim holds for a given $i<\max(\dom(\vec S))$, and let us prove it for $i+1$.
Let $\delta\in\cl(d^{i+1}_a)\setminus \cl(d^i_a)$ be arbitrary.
We have to verify that $\delta\notin T$.
By Clauses \eqref{C3ptree} and \eqref{C4ptree} of Definition~\ref{labeled-p-tree}, we may assume that $\delta\in\cl(d^{i+1}_a)\setminus d^{i+1}_a$.
In particular, as $d^{i+1}_a$ is the countable union of closed sets, we have $\cf(\delta)=\omega$.

\begin{subclaim} There exists a sequence $\langle\delta_n\mid n\in N\rangle$ of ordinals in $\delta$
such that:
\begin{itemize}
\item $N\in[\omega]^\omega$;
\item $\sup_{n\in N}\delta_n=\delta$;
\item for every $n\in N$, $n=\min\{{\bar n}<\omega\mid \delta_n\in S_{i+1}(p_{\bar n})\setminus S_i(p_{{\bar n}})\}$.
\end{itemize}
\end{subclaim}
\begin{proof}

Since $\delta\in\cl(d^{i+1}_a)\setminus(\cl(d^i_a)\cup d^{i+1}_a)$ and $\cf(\delta)=\omega$,
we may find a strictly increasing sequence $\langle\delta^m\mid m<\omega\rangle$
of ordinals in $d^{i+1}_a\setminus d^i_a$  such that $\sup_{m<\omega}\delta^m=\delta$.
For each $m<\omega$, let $n_m<\omega$ be the least such that $\delta^m\in S_{i+1}(p_{n_m})\setminus S_i(p_{n_m})$.
Since $S_{i+1}(p_n)$ is closed for every $n<\omega$, we get that $m\mapsto n_m$ is finite-to-one,
so that $N:=\{ n_m\mid m<\omega\}$ is infinite. For each $n\in N$,
set $m(n):=\min\{m<\omega\mid n=n_m\}$ and $\delta_n:=\delta^{m(n)}$.
Evidently,
\begin{gather*}
\min\{{\bar n}<\omega\mid \delta_n\in S_{i+1}(p_{\bar n})\setminus S_i(p_{{\bar n}})\}=\\
\min\{{\bar n}<\omega\mid \delta^{m(n)}\in S_{i+1}(p_{\bar n})\setminus S_i(p_{{\bar n}})\}=\\n_{m(n)}=n.
\end{gather*}

In particular, $\langle m(n)\mid n\in N\rangle$ is injective,
and $\sup_{n\in N}\delta_n=\delta$.
\end{proof}

Let $\langle\delta_n\mid n\in N\rangle$ be given by the subclaim.
By Definition~\ref{strategy}\eqref{C3pstrategy}, for all $n<m<\omega$, we have $(S_{i+1}(p_n)\setminus S_i(p_n))\sqsubseteq (S_{i+1}(p_m)\setminus S_i(p_m))$,
and hence $\delta=\sup_{n\in N}\sup(S_{i+1}(p_n)\setminus S_i(p_n))$.
Recalling that $S_{i}(p_n)\sqsubseteq S_{i+1}(p_n)$ for all $n<\omega$,
we conclude that $$\delta=\sup_{n\in N}\max(S_{i+1}(p_n)).$$

By Definition~\ref{labeled-p-tree}\eqref{C4ptree}, we have $(\max(S_{i+1}(p_m)),p_n)\in R$ for all $n\in N$ and $m\geq n$.
So, since, for each $m\in I$, $\dot{C}_{m}$ is a $\mathbb P_m$-name for a club,
we infer that  $(\delta,p_n)\in R$ for all $n\in N$.
Recalling the definition of $R$ and the fact that $I=\omega\setminus\lh(r^\star)$, we infer that,
for every $n\geq\min(N)$, $p_n\le r^\star$, and
$$p_n\Vdash_{\mathbb{P}_{n}}\check\delta\in \dot{C}_{n}.$$
Now, for every $n\geq\min(N)$, by the very choice of $\dot C_n$ and since $p_n\le r^\star$,
$p_n\Vdash_{\mathbb P_n}\dot T_n\cap\dot C_n=\emptyset$.
Altogether, for a tail of $n<\omega$,
$$p_n\Vdash_{\mathbb{P}_{n}}\check\delta\notin \dot{T}_{n}.$$
It thus follows from the definition of $\langle \dot{T_n}\mid n<\omega\rangle$ and the fact that $\{p_n\mid n<\omega\}\s\bar G$, that $\delta\notin T$.

Finally, suppose $i\in\acc^+(\dom(\vec S))$, and that the claim holds below $i$.
Let $\delta\in\cl(d^i_a)\setminus d^i_a$ be arbitrary.
By the previous analysis, it is clear that we may
pick $N\in[\omega]^\omega$  and an increasing sequence of ordinals $\langle \delta_n\mid n\in N\rangle$ that converges to $\delta$,
such that $\delta_n\in S_{i}(p_n)$ for all $n\in N$.
By the last clause of Definition~\ref{strategy}, for each $n\in N$,
we may let $j_n<i$ be the least for which there exists  $\delta'_n\in S_{j_n+1}(p_n)$ with $\delta_n\geq\delta'_n>\sup\{\delta_m\mid m\in N\cap n\}$.

If $\sup_{n\in N}j_n<i$, then by the induction hypothesis, $\delta\notin T$, and we are done.
Suppose that $\sup_{n\in N}j_n=i$. By thinning $N$ out, we may assume that $n\mapsto j_n$ is strictly increasing over $N$.
In particular, for all $m<n$ both from $N$, we have $\delta_m'\in S_{j_m+1}(p_m)\s S_{j_n}(p_m)\s S_{j_n}(p_n)\sqsubseteq S_{j_n+1}(p_n)$,
so that $\delta_m'\leq\max(S_{j_n}(p_n))\leq\delta_n'$.
Altogether,  $\delta=\sup_{n\in N}\max(S_{j_n}(p_n))$.
By Definition~\ref{labeled-p-tree}\eqref{C4ptree}, we have $(\max(S_{j_n}(p_m)),p_n)\in R$ whenever $n\in N$ and $m\in\omega\setminus n$.
Thus, as in the successor case, we have $(\delta,p_n)\in R$ for all $n\in N$, and hence  $\delta\notin T$.
\end{proof}

By appealing to Lemma~\ref{claim244}, we now fix a sequence $\langle a_\alpha\mid\alpha<\mu\rangle$
of conditions in $G$ such that, for all $\alpha<\mu$, letting $(p,\vec S):=a_\alpha$, we have $\dom(\vec S)=\alpha+1$.
Denote $D_\alpha:=\cl(d_{a_\alpha})$. By the preceding claim and regularity of $\mu$ we infer:\footnote{See Corollary~\ref{cor523}.}
\begin{claim}\label{c242} For every $\alpha<\mu$, $D_\alpha$ is a closed bounded subset of $\mu$, disjoint from $T$.\qed
\end{claim}
\begin{claim}\label{c241} For every $\alpha<\mu$ and $a'=(p',\vec{S'})$ in $G$ with $\dom(\vec{S'})=\alpha+1$,
$d_{a'}=d_{a_\alpha}$.
\end{claim}
\begin{proof} Denote $a_\alpha=(p,\vec S)$.
As $a_\alpha$ and $a'$ are in $G$, we may pick $(r,\vec T)$ that extends both.
In particular, $r\le p,p'$, and, for all  $q\in W(r)$,
$S_\alpha(w(p,q))=T_\alpha(q)=S'_\alpha(w(p',q))$.
Let $m:=\lh(r)-\lh(p)$. Then, for all $k<\omega$, $q\in W_k(r)\cap G$ iff $w(p,q)\in W_{m+k}(p)\cap G$. Note that these sets are singletons. Then
$$d_{a_\alpha}=\bigcup\{ S_{\alpha}(q)\mid q\in\bar G\cap W_{\geq m}(p)\}=\bigcup\{ T_{\alpha}(q)\mid q\in\bar G\cap W(r)\}.$$
Similarly, we have that $d_{a'}=\bigcup\{ T_{\alpha}(q)\mid q\in\bar G\cap W(r)\}$, and so $d_{a_\alpha}=d_{a'}$.
\end{proof}

\begin{claim}\label{claim243}  For every $\alpha<\beta<\mu$, $D_\alpha\sq D_\beta$.
\end{claim}
\begin{proof}  Let $\alpha<\beta<\mu$. It suffices to show that $d_{a_\alpha}\sq d_{a_\beta}$.
Let $(p,\vec S):=a_\beta$ and set $a:=(p,\vec S\restriction(\alpha+1))$. As $a_\beta\unlhd a$, we infer that $a\in G$.
Thus, the preceding claim yields $d_a=d_{a_\alpha}$.
Let $\langle p_n\mid n<\omega\rangle$ be the decreasing sequence of conditions such that
$p_n$ is unique element of $\bar G\cap W_n(p)$. Then:
\begin{itemize}
\item $d_{a_\alpha}=\bigcup\{ S_{\alpha}(p_n)\mid n<\omega\}$, and
\item $d_{a_\beta}=\bigcup\{ S_{\beta}(p_n)\mid n<\omega\}$.
\end{itemize}

Note that by Clauses \eqref{C3pstrategy} and \eqref{C5pstrategy} of Definition~\ref{strategy}, for all $n<\omega$, $S_\alpha(p_n)\sq S_\beta(p_n)$.
Now, let $\gamma<\mu$ be arbitrary. We consider two cases:

$\br$ If $\gamma\in d_{a_\alpha}$, then we may find $n<\omega$ such that $\gamma\in S_\alpha(p_n)$,
and as $S_\alpha(p_n)\sq S_\beta(p_n)$, we infer that $\gamma\in d_{a_\beta}$.

$\br$ If $\gamma\in d_{a_\beta}\setminus d_{a_\alpha}$,
then we first find $n<\omega$ such that $\gamma\in S_\beta(p_n)$.
In particular, $\gamma\in S_\beta(p_n)\setminus S_\alpha(p_n)$,
and as $S_\alpha(p_n)\sq S_\beta(p_n)$, this means that $\gamma\geq\sup(S_\alpha(p_n))$.
By Definition~\ref{labeled-p-tree}\eqref{C2ptree}, for all $m\geq n$, $S_\beta(p_n)\s S_\beta(p_m)$,
and so it likewise follows that, for all $m\geq n$, $\gamma\geq\sup(S_\alpha(p_m))$.
By Definition~\ref{labeled-p-tree}\eqref{C2ptree}, for all $m<n$, $S_\alpha(p_m)\s S_\alpha(p_n)$,
and so $\gamma\geq\sup(S_\alpha(p_n))\geq\sup(S_\alpha(p_m))$.
Altogether, $\gamma\geq\sup(d_{a_\alpha})$.
\end{proof}

\begin{claim}\label{claim584} For every $\epsilon<\mu$, there exists $\alpha<\mu$ such that $\max(D_\alpha)>\epsilon$.
\end{claim}
\begin{proof} By Lemma~\ref{claim244}, we may find $(q,\vec T)$ in $G$ and $\alpha>\epsilon$ such that, for all $r\in W(q)$, $\dom(\vec T)=\alpha+1$ and $\max(T_\alpha(r))=\alpha$.
By Claim~\ref{c241}, then, $\max(D_\alpha)=\alpha>\epsilon$.
\end{proof}
Put $D:=\bigcup\{D_\alpha\mid \alpha<\mu\}$.
By Claims \ref{c242} and \ref{claim243}, $D$ is closed subset of $\mu$, disjoint from $T$.
By Claim~\ref{claim584}, $D$ is unbounded.
So $T$ is nonstationary in $V[G]$.
\end{proof}

\begin{definition}\label{DeflenghtA}
Let $\lh_{\mathbb A}:=\lh\circ\pi$.
Denote $A_n:=\{ a\in A\mid \lh_{\mathbb A}(a)=n\}$,
$A^a_n:=\{ a'\in A\mid a'\unlhd a, \lh_{\mathbb A}(a')=\lh_{\mathbb A}(a)+n\}$,
and $\mathbb A_n:=(A_n\cup\{\one_{\mathbb A}\},\unlhd)$.
\end{definition}
\begin{definition}\label{DefCA}
Define $c_{\mathbb A}:A\rightarrow H_\mu$ by letting, for all $(p,\vec S)\in A$,
$$c_{\mathbb A}(p,\vec S):=(c(p),\{ ( i,c(q),S_i(q))\mid i\in\dom(\vec S), q\in W(p)\}).$$
\end{definition}

The rest of this section is devoted to verifying that $(\mathbb A,\lh_{\mathbb A},c_{\mathbb A})$ is a $\Sigma$-Prikry forcing that admits a forking projection to $(\mathbb P,\lh,c)$.

\begin{definition}[Projection and forking]\label{d45}\hfill
\begin{itemize}
\item Define $\pi:A\rightarrow P$ by stipulating $\pi(p,\vec S):=p$.
\item Given $a=(p,\vec S)$ in $A$, define $\fork{}{a}:\cone{p}\rightarrow A$ by letting
for each $p'\le p$, $\fork{}{a}(p'):=(p',\vec{S'})$, where $\vec{S'}$ is the sequence $\langle S_i':W(p')\rightarrow[\mu]^{<\mu}\mid i<\dom(\vec{S})\rangle$ to satisfy:
\begin{equation}\label{pitchfork}
\tag{*}S'_i(q):=S_i(w(p,q))\text{  for all }i\in\dom(\vec{S'})\text{ and }q\in W(p').
\end{equation}
\end{itemize}
\end{definition}

\begin{lemma}\label{pitchforkWellDefined} Let $a\in A$ and $p'\le\pi(a)$.
Then $\fork{}{a}(p')\in A$ and $\fork{}{a}(p')\unlhd a$, so that $\fork{}{a}$ is a well-defined function from $\cone{\pi(a)}$ to $\conea{a}$.
\end{lemma}
\begin{proof} Set $a:=(p,\vec{S})$. If $\vec S=\emptyset$, then $\fork{}{a}(p')=\myceil{p'}{\mathbb A}$, and we are done.

Next, suppose that $\dom(\vec S)=\alpha+1$. Let $(p',\vec{S'}):=\fork{}{a}(p')$.
Let $i\leq\alpha$ and we shall verify that $S'_i$ is a $p'$-labeled tree. To this end, let $q'\le q$ be arbitrary pair of elements of $W(p')$.
\begin{itemize}
\item By Definition~\ref{SigmaPrikry}(\ref{itsaprojection}),
we have $w(p,q')\le w(p,q)$, so that $S_i'(q')=S_i(w(p,q'))\supseteq S_i(w(p,q))=S_i'(q)$.
\item As $q\le w(p,q)$, $w(p,q)\Vdash_{\mathbb P} S_i(w(p,q))\cap\dot{T}=\emptyset$, so that, since $S_i'(q)=S_i(w(p,q))$,
we clearly have $q\Vdash_{\mathbb P} S_i'(q)\cap\dot{T}=\emptyset$.
\item  To avoid trivialities, suppose that $S'_i(q')\neq\emptyset$. Write $\gamma:=\max(S_i(w(p,q))$.
As $(\gamma,w(p,q))\in R$ and $q\le w(p,q)$, we clearly have $(\gamma,q)\in R$. Recalling that $\max(S_i'(q))=\gamma$, we are done.
\end{itemize}
To prove that $(p',\vec{S}')$ is a condition in $A$ it remains to argue that $\vec{S}'$ fulfills the requirements described in Clauses~\eqref{C3pstrategy} and \eqref{C5pstrategy} of Definition~\ref{strategy} but this already follows from the definition of  $\vec{S}'$ and the fact that $\vec{S}$ is a $p$-strategy. Finally $\fork{}{a}(p')=(p',\vec{S}')\unlhd (p,\vec{S})=a$ by the very choice of $p'$ and by Definition~\ref{d45}.
\end{proof}

Let us now check that the pair of functions $(\pitchfork,\pi)$ of Definition~\ref{d45} is a forking projection from $(\mathbb{A},\ell_\mathbb{A}, c_\mathbb{A})$ to $(\mathbb{P},\ell, c)$. We prove this by going over the clauses of Definition~\ref{forking}.
\begin{lemma}\label{forkingindeed}
\begin{enumerate}
\item\label{C1forkingindeed} $\pi$ is a projection from $\mathbb A$ onto $\mathbb P$, and $\lh_{\mathbb A}=\lh\circ\pi$;
\item\label{C2forkingindeed} for all $a\in A$, $\fork{}{a}$ is an order-preserving function from $(\cone{\pi(a)},\le)$ to $(\conea{a},\unlhd)$;
\item\label{C3forkingindeed} for all $p\in P$, $(p,\emptyset)$ is the greatest element of $\{ a \in A\mid \pi(a)=p\}$;
\item\label{C4forkingindeed} for all $n,m<\omega$ and $b\unlhd^{n+m} a$, $m(a,b)$ exists and satisfies:
$$m(a,b)=\fork{}{a}({m(\pi(a),\pi(b))});$$ 
\item\label{C5forkingindeed} for all $a\in A$ and $p'\le\pi(a)$, $\pi(\fork{}{a}(p'))=p'$;
\item\label{C6forkingindeed} for all $a\in A$ and $p'\le\pi(a)$, $a=(\pi(a),\emptyset)$ iff $\fork{}{a}(p')=(p',\emptyset)$;
\item\label{C7forkingindeed} for all $a\in A$, $a'\unlhd a$ and $r\le \pi(a')$, $\fork{}{a'}(r)\unlhd\fork{}{a}(r)$;
\item\label{C8forkingindeed} for all $a,a'\in A$, if $c_{\mathbb A}(a)=c_{\mathbb A}(a')$, then $c(\pi(a))=c(\pi(a'))$ and $\fork{}{a}(r)=\fork{}{a'}(r)$ for every $r\le \pi(a),\pi(a')$.
\end{enumerate}
\end{lemma}
\begin{proof}
\begin{enumerate}
\item The equality between the lengths comes from Definition \ref{DeflenghtA} so let us concentrate on proving that $\pi$ forms a projection. 
Clearly, $\pi(\one_{\mathbb A})=\one_{\mathbb P}$.
By Definition~\ref{d20}, for all $a'\unlhd a$ in $A$, we have  $\pi(a')\le\pi(a)$.
Finally, suppose that $a\in A$ and $p'\le\pi(a)$, and let us find $a'\unlhd a$ such that $\pi(a')\le p'$.
Put $a':=\fork{}{a}(p')$. Then it is not hard to check that $a'\unlhd a$ and $\pi(\fork{}{a}(p'))=p'$, so we are done.

\item Let $a=(p,\vec{S})$ be an arbitrary element of $A$.
By Lemma~\ref{pitchforkWellDefined}, $\fork{}{a}$ is a function from $\cone{\pi(a)}$ to $\conea{a}$.
To see that it is order-preserving, fix $r\le q$ below $\pi(a)$.
By Definition~\ref{d45}, $\fork{}{a}(r)=(r, \vec{R})$ and $\fork{}{a}(q)=(q, \vec{Q})$, where $\vec{R}$ and $\vec{Q}$ are as described in Definition~\ref{d45}(\ref{pitchfork}).
In particular, $\dom(\vec{R})=\dom(\vec{S})=\dom(\vec{Q})$. So, to establish that $\fork{}{a}(r)\unlhd \fork{}{a}(q)$, it suffices to verify Clause~\eqref{C2cd20} of Definition~\ref{d20}.
Let $i\in \dom(\vec{R})$ and $r'\in W(r)$ be arbitrary and notice that $\eqref{pitchfork}$ implies $R_i(r')=S_i(w(p,r'))$.
Since $r\le q$, hence $w(q,r')\in W(q)$, again by $\eqref{pitchfork}$, $Q_i(w(q,r'))=S_i(w(p,w(q,r')))$.
Using Lemma~\ref{l15}, it is the case that $Q_i(w(q,r'))=S_i(w(p,r'))$, hence $R_i(r')=Q_i(w(q,r'))$.

\item This is easy to see.

\item Write $a=(p,\vec S)$ and $b=(\bar{p},\vec{T})$.
Appealing to Definition~\ref{SigmaPrikry}\eqref{c5}, set $p':=m(p,\bar{p})$,
so that $\bar{p}\le^m p'\le^n p$.
Now, let $a':=\fork{}{a}(p')$.
By Definition~\ref{d45}, $a'$ takes the form $(p',\vec{S'})$,
where $\dom(\vec{S'})=\dom(\vec{S})$, and
$S'_i(q):=S_i(w(p,q)),\text{  for all }i\in\dom(\vec{S'})\text{ and }q\in W(p').$ Observe that if we prove $a'=m(a,b)$, i.e.,
that $a'$ is the greatest element of $\{c\in A^a_{n}\mid  c\in A^b_m\}$, we will be done with both assertions.

\begin{claim} $a'$ belongs to $\{c\in A^a_{n}\mid  c\in A^b_m\}$.
\end{claim}
\begin{proof} By Clauses \eqref{C1forkingindeed} and \eqref{C2forkingindeed} together with Clause~\eqref{C5forkingindeed} below, 
$a'$ is an element of $A^a_n$, so it suffices to show that $b\unlhd a'$.

We already know that $\bar{p}\le^m p'$ and $\dom(\vec{T})\geq\dom(\vec{S})=\dom(\vec{S'})$,
thus, by virtue of Definition~\ref{d20}, we are left with verifying that $T_i(q)=S_i'(w(p',q))$
for all $i\in \dom(\vec{S'})$ and $q\in W(\bar{p})$.

Let $i$ and $q$ be as above.
As $b\unlhd a$, we infer that $T_i(q)=S_i(w(p,q))$. 
By definition of $S_i'$ and Lemma~\ref{l15},
$S'_i(w(p',q))=S_i(w(p,w(p',q))=S_i(w(p,q))$,
so that, altogether, $T_i(q)=S'_i(w(p',q))$, as desired.
\end{proof}

\begin{claim}
$a'$ is the greatest element of $\{c\in A^a_{n}\mid  b\in A^b_m\}$. 
\end{claim}
\begin{proof}
Let $c=(r,\vec R)$ be a condition with $(\bar p,\vec{T})\unlhd^m (r,\vec{R})\unlhd^n (p,\vec{S})$.
In particular, $\bar p\le^m r\le^n p$, so that, since $p'=m(p,\bar p)$, $r\le^0 p'$.

We already know that $r\le p'$ and $\dom(\vec{R})\geq\dom(\vec{S})=\dom(\vec{S'})$.
Now, let $i\in \dom(\vec{S'})$ and $q\in W(r)$ be arbitrary.
By definition of $S_i'$ and Lemma~\ref{l15},
$S'_i(w(p',q))=S_i(w(p,w(p',q))=S_i(w(p,q))$.
As $c\unlhd a$, the latter is equal to $R_i(q)$, hence $c\unlhd a'$, as desired.
\end{proof}

\item This follows immediately from Definition~\ref{d45}.

\item Suppose that $a\in A$ with $a=(\pi(a),\emptyset)$.
By Definition~\ref{d45}(\ref{pitchfork}), for all $p'\le \pi(a)$, $\fork{}{a}(p')=(p',\emptyset)$.
Conversely, let $a:=(\pi(a), \vec{S})$ and suppose that $\fork{}{a}(q)=(q,\emptyset)$. 
Again, by Definition~\ref{d45}, $\dom(\vec{S})=\emptyset$,  and thus $a=(\pi(a),\emptyset)$, as desired.

\item  Let $a\in A$, $a'\unlhd a$ and $r\le \pi(a')$ be arbitrary,
say $a'=(p', \vec{S'})$ and $a=(p, \vec S)$. By Definition~\ref{d20}, the following three hold:
\begin{itemize}
\item $p'\le p$;
\item $\dom(\vec{S})\leq\dom(\vec{S}')$,
\item $S'_i(q)=S_i(w(p,q))$, for all $i\in \dom(\vec{S})$ and $q\in W(p')$.
\end{itemize}
By Definition~\ref{d45}, $\fork{}{a}(r):=(r,\vec{S}^a)$, where $\dom(\vec{S}^a)=\dom(\vec{S})$ and,
for all $i<\dom(\vec{S})$ and $q\in W(r)$, $S^a_i(q)=S_i(w(p,q))$.
A similar statement is valid for $\fork{}{a'}(r)=(r,\vec{S}^{a'})$.
Notice that $\dom(\vec{S}^{a'})\geq\dom(\vec{S}^{a})$ and that,
for all $i<\dom(\vec{S}^a)$ and $q\in W(r)$,
Lemma~\ref{l15} yields the following chain of equalities:
$$S^{a'}_i(q)=S'_i(w(p',q))=S_i(w(p, w(p',q)))=S_i(w(p,q))=S^a_i(q).$$
Altogether we have proved $\fork{}{a'}(r)\unlhd \fork{}{a}(r)$.
\item  Let $a=(p,\vec{S})$ and $a'=(p',\vec{S}')$ be elements of $A$ with $c_{\mathbb A}(a)=c_{\mathbb A}(a')$.
By Definition~\ref{DefCA}, then, $c(\pi(a))=c(\pi(a'))$ and $\dom(\vec{S})=\dom(\vec{S}')$.
Now, let $r\le \pi(a),\pi(a')$ be arbitrary; we shall show that $\fork{}{a}(r)=\fork{}{a'}(r)$.
Recall that $\fork{}{a}(r)=(r,\vec{T})$ and $\fork{}{a'}(r)=(r,\vec{T}')$, where $\vec{T}$ and $\vec{T}'$ are the $r$-strategy of length $\dom(\vec{S})$ given by Definition~\ref{d45}\eqref{pitchfork} with respect to $a$ and $a'$, respectively.
Therefore, it suffices to show that, for all $i\in\dom(\vec{S})$ and $q\in W(r)$, $S_i(w(p,q))=S'_i(w(p',q))$.
Let $i\in \dom(\vec{S})$ and  $q\in W(r)$ be arbitrary.
By Lemma~\ref{lemma7}(4),  $c\restriction W(p)$ is injective. Since $c_{\mathbb A}(a)=c_{\mathbb A}(a')$, Definition~\ref{DefCA} yields $c``W(p)=c``W(p')$.
Consequently, $c(w(p,q))=c(t)$, where $t$ is the unique element of $W(p')$ that is compatible with $w(p,q)$ and has the same length. Thus, it is not hard to check that $t=w(p',q)$, hence $c(w(p,q))=c(w(p',q))$.
Finally, as $c_\mathbb{A}(a)=c_\mathbb{A}(a')$ and  $c(w(p,q))=c(w(p',q))$, it is the case that $S_i(w(p,q))=S'_i(w(p',q))$.\qedhere
\end{enumerate}
\end{proof}
\begin{remark}
Note that the above proof only uses the fact that the triple $(\mathbb{P},\lh,c)$ is $\Sigma$-Prikry 
together with the defining properties of $(\mathbb{A},\ell_\mathbb{A},c_\mathbb{A})$ (that is, Definitions \ref{d20}, \ref{DeflenghtA}, \ref{DefCA} and \ref{d45}).
In particular, we have not relied on any clause of Definition~\ref{SigmaPrikry} for $(\mathbb{A},\ell_\mathbb{A},c_\mathbb{A})$, which have not yet been verified.
\end{remark}

\begin{lemma}\label{C1ASigmaPrirky} Let $n<\omega$.
Suppose that $D$ is a directed family of conditions in $\mathbb A_n$,
$|D|<\kappa_n$, and for some $\bar p$, we have $\pi(a)=\bar p$ for all $a\in D$.
Then $D$ admits a lower bound in $\mathbb A_n$.
\end{lemma}
\begin{proof} Since $D$ is directed, given any $a,a'\in D$, we may pick $b\in D$ extending $a$ and $a'$;
now, as $\pi[D]=\{\bar p\}$, find $\vec S,\vec S',\vec T$ such that $a=(\bar p,\vec S)$, $a'=(\bar p,\vec S')$ and $b=(\bar p,\vec T)$,
and note that, by Definition~\ref{d20}, for all $q\in W(\bar p)$ and $i\in\dom(\vec{S})\cap \dom(\vec{S'})$, $S_i(q)=\vec T_i(q)=S'_i(q)$.
It thus follows that $D$ is linearly ordered by $\unlhd$, and, for all $(\bar p,\vec S),(\bar p,\vec S')\in D$,
$(\bar p,\vec S)\unlhd (\bar p,\vec S')$ iff $\dom(\vec S)\geq\dom(\vec S')$.
So $(D,\unlhd)$ is order-isomorphic to $(\theta,\ni)$ for some ordinal $\theta<\kappa_n$.
In particular, if $\theta$ is a successor ordinal, then $D$ admits a lower bound. So let us assume that $\theta$ is a limit ordinal.

For every $\tau<\theta$, let $(\bar p,\vec S^\tau)$ denote the $\tau^{th}$-element of $D$.
Set $\alpha:=\sup_{\tau<\theta}\dom(\vec{S^\tau})$.
We define a $\bar{p}$-strategy $\vec{S}=\langle S_i\mid i\leq\alpha\rangle$ as follows. Fix $q\in W(\bar p)$.

$\br$ For $i<\alpha$, $S_i(q)$ is defined as the unique element of $\{S^\tau_i(q)\mid \tau<\theta, i\in\dom(\vec{S^\tau})\}$.

$\br$ For $i=\alpha$, we distinguish two cases:

$\br\br$ If $S_i(q)=\emptyset$ for all $i<\alpha$, then we continue and let $S_\alpha(q):=\emptyset$;

$\br\br$ Otherwise, let $S_\alpha(q):=\bigcup_{i<\alpha}S_i(q)\cup\{\beta_q\}$, where $$\beta_q:=\sup\{\max(S_i(q))\mid i<\alpha, S_i(q)\neq\emptyset\}.$$

\begin{claim} $(\bar p,\vec S)\in A_n$. In particular, $(\bar p,\vec S)$ is a lower bound for $D$.
\end{claim}
\begin{proof}
Since, for each $\tau<\theta$, $\vec{S}^\tau$ is a $\bar{p}$-strategy, a moment of reflection makes it clear that we only need to verify that $S_\alpha$ is a labeled $\bar p$-tree. Let $q\in W(\bar p)$ be arbitrary. As $\langle S_i(q)\mid i<\alpha\rangle$ is weakly $\sqsubseteq$-increasing sequence of closed sets we only need to verify Clauses~\eqref{C3ptree} and \eqref{C4ptree} of Definition~\ref{labeled-p-tree}.
First we show that $q\Vdash_{\mathbb{P}}S_\alpha(q)\cap \dot{T}=\emptyset$.
For this aim observe that Definition~\ref{labeled-p-tree}\eqref{C4ptree} yields $(q,\max(S_i(q))\in R$,  for each $i<\alpha$. Now, for each $r\leq q$ with $\ell(r)\in I$ and $i<\alpha$, $r\Vdash_{\mathbb{P}_{\ell(r)}} \max(S_i(q))\in \dot{C}_{\ell(r)}$, hence $r\Vdash_{\mathbb{P}_{\ell(r)}} \beta_q\in \dot{C}_{\ell(r)}$, and thus, again by definition of $R$, $(\beta_q, q)\in R$ (cf. Lemma~\ref{lemma9}). Combining  Definition~\ref{labeled-p-tree}\eqref{C3ptree} with $(\beta_q, q)\in R$ it altogether follows that $q\Vdash_{\mathbb{P}}S_\alpha(q)\cap \dot{T}=\emptyset$.

Finally let $q'\le q$ and let us check that the last bullet holds. For all $i<\alpha$, since $S_i$ is a $\bar{p}$-strategy, either $S_i(q')=\emptyset$ or $(\max(S_i(q')),q)\in R$.
If $S_\alpha(q')\neq\emptyset$,
then $\max(S_\alpha(q'))$ is the limit of $\langle \max(S_i(q'))\mid i<\alpha, S_i(q')\neq\emptyset\rangle$, so that, arguing as before, $(\max(S_\alpha(q')),q)\in R$.

Thus we have shown that $(\bar{p},\vec{S})\in A_n$ and clearly $(\bar{p},\vec{S})$ gives a lower bound for $D$.
\end{proof}

This completes the proof.
\end{proof}

\begin{lemma}[Mixing property]\label{mixinglemma}
Let $(p,\vec S)=a\in A$, $p'\le^0 p$, and $m<\omega$.
Suppose that $g:W_m(p')\rightarrow\conea{a}$ is a function such that $\pi\circ g$ is the identity map.
Then there exists $b\unlhd^0 a$ with $\pi(b)=p'$ such that $\fork{}{b}(r)\unlhd^0 g(r)$ for every $r\in W_m(p')$.
\end{lemma}
\begin{proof} Using Definition~\ref{SigmaPrikry}(\ref{csize}), we may find some cardinal $\theta<\mu$ and an injective enumeration $\{r^\tau\mid \tau<\theta\}$ of $W_m({p'})$.
For each $\tau<\theta$, let $\vec{S}^\tau$ be such that $g(r^\tau)=(r^\tau, \vec{S}^\tau)$.
As we are seeking $b\unlhd^0 a$ such that, in particular, for every $\tau<\theta$, $\fork{}{b}(r)\unlhd^0 g(r^\tau)$,
we may make our life harder and assume that $\dom(\vec{S}^\tau)$ is nonzero, say $\dom(\vec{S}^\tau)=\alpha_\tau+1$.

Set $\alpha:=\sup(\dom(\vec{S}))$, so that, if $\dom(\vec S)>0$, then $\dom(\vec S)=\alpha+1$.
Set $\alpha':=\sup_{\tau<\theta}\alpha_\tau$, and note that, by regularity of $\mu$, $\alpha\leq\alpha'<\mu$.
Our goal is to define a sequence $\vec{T}=\langle T_i:W(p')\rightarrow[\mu]^{<\mu}\mid i\leq\alpha'\rangle$ for which $b:=(p',\vec{T})$ satisfies the conclusion of the lemma.

As $\{r^\tau\mid \tau<\theta\}$ is an enumeration of the $m^{th}$-level of the $p$-tree $W(p')$,
Lemma~\ref{lemma7} entails that, for each $q\in W({p'})$, there is a unique ordinal $\tau_q<\theta$, such that $q$ is comparable with $r^{\tau_q}$.
It thus follows from Lemma~\ref{lemma7}(3) that, for all $q\in W(p')$, $\ell(q)-\ell(p')\geq m$ iff $q\in W(r^{\tau_q})$.

Now, for all $i\leq\alpha'$ and $q\in W({p'})$, let:
$$T_i(q):=
\begin{cases}
S^{\tau_q}_{\min\{i,\alpha_{\tau_q}\}}(q),& \text{if }q\in W(r^{\tau_q});\\
S_{\min\{i,\alpha\}}(w(p,q)),& \text{if }q\notin W(r^{\tau_q})\text{ and }\alpha>0;\\
\emptyset,&\text{otherwise}.
\end{cases}$$

\begin{claim}\label{claim5211}
Let $i\leq\alpha'$. Then $T_i$ is a labeled $p'$-tree.
\end{claim}
\begin{proof} 
Fix $q\in W(p')$ and let us go over the Clauses of Definition~\ref{labeled-p-tree}.
\begin{enumerate}
\item It is clear that in any of the three cases, $T_i(q)$ is a closed bounded subset of $\mu$.

\item Let $q'\le q$. 
We focus on the non-trivial case in which $\ell(q')-\ell(p')\geq m$, while $\ell(q)-\ell(p')<m$ and $\alpha>0$.

$\br$ If $i\leq \alpha$, then $T_i(q)=S_i(w(p,q))$ and $T_i(q')=S^{\tau_q}_i(q')$.
In this case, since $w(r^{\tau_q}, q)\le w(p,q)$ and $\vec{S}$ is a $p$-strategy, $S_i(w(p,q))\s S_i(w(r^{\tau_q},q))$. 
In addition, since $(r^{\tau_q}, \vec{S}^{\tau_q})\unlhd (p,\vec{S})$, $S_i(w(r^{\tau_q},q))=S^{\tau_q}_i(q)$, so that $T_i(q)\s S^{\tau_q}_i(q)$. 
But $S^{\tau_q}_i(q)\s S^{\tau_q}_i(q')$, so that altogether $T_i(q)\s T_i(q')$, as desired.

$\br$ If $i>\alpha$, then $T_i(q)=S_{\alpha}(w(p,q))$ and $T_i(q')=S^{\tau_q}_j(q')$ for $j:=\min\{i,\alpha_{\tau_q}\}$. 
In this case, as $\vec{S}$ is a $p'$-strategy and $\vec{S}^{\tau_q}$ is an $r^{\tau_q}$-strategy,
we infer from  $(r^{\tau_q}, \vec{S}^{\tau_q})\unlhd (p,\vec{S})$ that:
$$S_{\alpha}(w(p,q))\s S_{\alpha}(w(r^{\tau_q},q))=S^{\tau_q}_{\alpha}(q)\sq S^{\tau_q}_{j}(q)\s S^{\tau_q}_{j}(q').$$
Altogether, $T_i(q)\s T_i(q')$, as desired.

\item If $q\in W(r^{\tau_q})$, then this follows from the fact that $S^{\tau_q}_{\min\{i,\alpha_{\tau_q}\}}$
is a labeled $r^{\tau_q}$-tree. 
If $q\notin W(r^{\tau_q})\text{ and }\alpha>0$, then this follows from the fact that $S_{\min\{i,\alpha\}}$ is a labeled $p$-tree and $q\le w(p, q)$.

\item Let $q'\le q$ in $W(p')$ and assume that $T_i(q')\neq \emptyset$. 
We focus on the case $T_i(q')=S_j(w(p,q'))$, for $j:=\min\{i,\alpha\}$.
In particular, $\beta:=\max(S_j(w(p,q')))$ is well-defined.
Clearly $w(p,q')\le w(p,q)$ so, since $S_j$ is a labeled $p$-tree, $(\beta, w(p,q'))\in R$.
But $q'\leq w(p,q')$, so by the nature of $R$, we have that $(\beta, q')\in R$, as well.\qedhere
\end{enumerate}
\end{proof}

\begin{claim}
The sequence $\vec{T}=\langle T_i:W(p')\rightarrow[\mu]^{<\mu}\mid i\leq\alpha'\rangle$ is a $p'$-strategy.
\end{claim}
\begin{proof} We need to go over the clauses of Definition~\ref{strategy}.
However, Clause~\eqref{C1pstrategy} is trivial, Clause~\eqref{C2pstrategy} is established in the preceding claim,
and Clauses~\eqref{C3pstrategy} and \eqref{C5pstrategy} follow from the corresponding features of $\vec S$ and the $\vec S^\tau$'s.
Thus, we are left with verifying Clause~\eqref{C4pstrategy}.

To this end, fix $i<\alpha$ and a pair $q'\le q$ in $W(p')$. We have to show that $(T_{i+1}(q)\setminus T_i(q))\sqsubseteq (T_{i+1}(q')\setminus T_i(q'))$. 
As before, the only non-trivial case is when $\ell(q')-\ell(p')\geq m$, while $\ell(q)-\ell(p')<m$ and $\alpha>0$.
To avoid arguing about the empty set, we may also assume that $\alpha>i$. In particular, $\alpha_\tau>i$. So
\begin{itemize}
\item $T_{i+1}(q)\setminus T_i(q)=S_{i+1}(w(p,q))\setminus S_{i}(w(p,q))$, and
\item $T_{i+1}(q')\setminus T_i(q')=S^{\tau_q}_{i+1}(q')\setminus S^{\tau_q}_{i}(q')$.
\end{itemize}

Now, as $\vec{S}$ is a $p$-strategy, we infer that $S_{i+1}(w(p,q))\setminus S_i(w(p,q))\sq S_{i+1}(w(p,q'))\setminus S_i(w(p,q'))$. 
But $(r^{\tau_{q'}}, \vec{S}^{\tau_{q'}})\unlhd (p, \vec{S})$, and hence, for each $j\in\{i,i+1\}$, $S^{{\tau_{q'}}}_j(q')=S_j(w(p,q'))$. 
The desired equation now follows immediately.
\end{proof}

Thus, we have established that $b:=(p',\vec{T})$ is a legitimate condition.

\begin{claim}\label{claim5123} $\pi(b)=p'$ and $b\unlhd^0 a$.
\end{claim}
\begin{proof}
The first assertion is trivial, and it also implies that $b\unlhd^0 a$ iff  $b\unlhd a$,
hence, we focus on establishing the latter.
As $p'\le p$ and $\alpha'\geq\alpha$, we are left with verifying Clause~\eqref{C2cd20} of Definition~\ref{d20}.
To avoid trivialities, suppose also that $\alpha>0$.
Now, let $i\leq\alpha$ and $q\in W(p')$ be arbitrary. 

$\br$ If $\lh(q)<\lh(p')+m$, then we have $T_i(q)=S_i(w(p,q))$, and we are done. 

$\br$ If $\lh(q)\geq\lh(p')+m$, then $T_i(q)=S_i^{\tau_q}(q)$ and, since $(r^{\tau_q}, \vec{S}^{\tau_q})\unlhd (p,\vec{S})$, $T_i(q)=S_i(w(p,q))$, as desired.
\end{proof}

\begin{claim} Let $\tau<\theta$. For each $q\in W(r^\tau)$, $w(p',q)=w(r^\tau,q)$.
\end{claim}
\begin{proof}  As $r^\tau\le p'$, we have $\{ s\mid q\le s\le r^\tau\}\s\{ s\mid  q\le s\le p'\}$,
so that $w(r^\tau,q)\le w(p', q)$.
In addition, as $w(p',q)$ and $r^\tau$ are compatible elements of $W(p')$ (as witnessed by $q$),
we infer from Lemma~\ref{lemma7}(2), $\lh(w(p',q))=\lh(q)\geq\lh(r^\tau)$ and Definition~\ref{SigmaPrikry}(\ref{c4}),
that $w(p',q)\le r^\tau$, so that $w(p',q)\in \{ s\mid q\le s\le r^\tau\}$,
and hence $w(p',q)\le w(r^\tau, q)$.
\end{proof}

Recalling Claim~\ref{claim5123},
to complete our proof, we fix an arbitrary $\tau<\theta$, and turn to show that $\fork{}{b}(r^\tau)\unlhd^0 g(r^\tau)$.
By Lemma~\ref{forkingindeed}\eqref{C5forkingindeed}, $\pi(\fork{}{b}(r^\tau))=r^\tau=\pi(g(r^\tau))$,
so that we may focus on verifying that $\fork{}{b}(r^\tau)\unlhd g(r^\tau)$.

To this end, let $\vec T^\tau$ denote the $r^\tau$-strategy such that $\fork{}{b}(r^\tau)=(r^\tau,\vec{T}^\tau)$.
By Definition~\ref{d45}\eqref{pitchfork}, $\dom(\vec{T}^\tau)=\dom(\vec T)=\alpha'+1$, hence $\dom(\vec{S}^\tau)=\alpha_\tau+1\leq\alpha'+1\leq\dom(\vec{T}^\tau)$. 
Now, let $i\leq \alpha_\tau$ and $q\in W(r^\tau)$.
By Definition~\ref{d45}\eqref{pitchfork}, $T^\tau_i(q)=T_i(w(p',q))$. 
By the preceding claim $w(p',q)=w(r^\tau,q)$, so that $q':= w(p',q)$ is in $W(r^\tau)$ and $\tau_{q'}=\tau$.
In effect, by definition of $T_i(q')$ (just before Claim~\ref{claim5211}, we get that $T_i(q')=S^\tau_i(q')$.
Altogether, $T_i^\tau(q)=S^\tau_i(q')=S_i^\tau(w(r^\tau,q'))$, as required by Clause~\eqref{C2cd20} of Definition~\ref{d20}.
\end{proof}

\begin{cor}\label{corA}
$(\mathbb{A},\lh_\mathbb{A}, c_\mathbb{A})$ is a $\Sigma$-Prikry triple,
and $\one_{\mathbb A}\Vdash_{\mathbb{A}}\check{\mu}=\check\kappa^+$.
\end{cor}
\begin{proof} We first go over the clauses of Definition~\ref{SigmaPrikry}:
\begin{enumerate}
\item By Lemma~\ref{C2ASigmaPrikry}.
\item By Lemma~\ref{313} together with Lemma~\ref{C1ASigmaPrirky}.
\item By Lemma~\ref{C3ASigmaPriky} and the fact that $|H_\mu|=\mu$.
\item By Lemma~\ref{lemma48}.
\item By Lemma~\ref{lemma49}.
\item By Lemma~\ref{lemma410}.
\item By Lemma~\ref{corollary320} together with Lemma~\ref{mixinglemma}.
\end{enumerate}

Finally, by Corollary~\ref{cor523} and the fact that $\one_{\mathbb P}\Vdash_{\mathbb P}``\check\kappa\text{ is singular}"$,
$\one_{\mathbb A}\Vdash_{\mathbb{A}}\check{\mu}=\check\kappa^+$.
\end{proof}

For the record, we make explicit one more feature of the forking projection constructed in this section.

\begin{lemma}[Transitivity]\label{claimRowbottom}
Let $a\in A$. For all $q\le\pi(a)$ and $r\in W(q)$, $$\fork{}{a}(r)=\fork{}{\fork{}{a}(q)}(r).$$
\end{lemma}
\begin{proof} Set $(p,\vec{S}):=a$.
Fix an arbitrary $q\le\pi(a)$, and let $b=\fork{}{a}(q)$.
Fix an arbitrary $r\in W(q)$, and set $(t, \vec{T}):=\fork{}{a}(r)$ and $(u,\vec{U}):=\fork{}{b}(r)$.
By Definition~\ref{d45}, it follows that $t=r=u$ and $\dom(\vec{T})=\dom(\vec{S})=\dom(\vec{U})$. Once again Definition~\ref{d45} yields, for each $i\in\dom(\vec{S})$ and $s\in W(t)$, $T_i(s)=S_i(w(p,s))$. Analogously, for each $i\in\dom(\vec{S})$ and $s\in W(u)$, $Q_i(s)=S_i(w(p,s))$. Altogether, $W(t)=W(u)$, and for each $i\in\dom(\vec{S})$ and $s\in W(u)$, $T_i(s)=Q_i(s)$, as desired.
\end{proof}

\section{Conclusion}
By putting everything together, we arrive at the following corollary:

\begin{cor}\label{onestep}  Suppose
$\Sigma=\langle \kappa_n\mid n<\omega\rangle$ is a non-decreasing sequence of Laver-indestructible supercompact cardinals,
and let $\kappa:=\sup(\Sigma)$. Suppose:
\begin{itemize}
\item[(i)] $(\mathbb P,\lh,c)$ is a $\Sigma$-Prikry notion of forcing, and $\one_{\mathbb P}\Vdash_{\mathbb{P}}``\check\kappa\text{ is singular}"$;
\item[(ii)] $\one_{\mathbb P}\Vdash_{\mathbb{P}}\check\mu=\check\kappa^+$, for some cardinal $\mu=\mu^{<\mu}$;
\item[(iii)] $\mathbb P\s H_{\mu^+}$;
\item[(iv)] $r^\star\in P$ forces that $z$ is a $\mathbb P$-name for a stationary subset of $(E^\mu_\omega)^V$
that does not reflect in $\{\alpha<\mu\mid \omega<\cf^V(\alpha)<\kappa\}$.
\end{itemize}

Then, there exists a $\Sigma$-Prikry triple $(\mathbb A,\lh_{\mathbb A},c_{\mathbb A})$ such that:
\begin{enumerate}
\item $(\mathbb A,\lh_{\mathbb A},c_{\mathbb A})$ admits a forking projection to $(\mathbb P,\lh,c)$ 
that has the mixing property;
\item $\one_{\mathbb A}\Vdash_{\mathbb A}\check\mu=\check\kappa^+$;
\item $\mathbb A\s H_{\mu^+}$;
\item $\myceil{r^\star}{\mathbb A}$ forces that $z$ is nonstationary.
\end{enumerate}
\end{cor}
\begin{proof}  By Lemma~\ref{l25}, for all $n<\omega$,
$V^{\mathbb P_n}\models \refl({<}\omega,E^\mu_{<\kappa_n},E^\mu_{<\kappa_n})$.
So, all the blanket assumptions of Section~\ref{killingone} are satisfied,
and we obtain a notion of forcing $\mathbb A:=\mathbb A(\mathbb P,z)$ together with maps $\lh_{\mathbb A}$ and $c_{\mathbb A}$
such that, by Corollary~\ref{corA}, $(\mathbb A,\lh_{\mathbb A},c_{\mathbb A})$ is $\Sigma$-Prikry.

Now, Clauses (1) and (2) follow from Lemma~\ref{forkingindeed} and Corollary~\ref{corA},
Clause~(3) follows from Lemma~\ref{lemma35}, and
Clause~(4) follows from Theorem~\ref{purpose}.
\end{proof}

\newcommand{\etalchar}[1]{$^{#1}$}

\end{document}